\documentclass{article}
\usepackage[utf8]{inputenc}
\usepackage[T1]{fontenc}
\usepackage[margin=1.2in]{geometry}
\usepackage{lmodern}
\usepackage{float,subcaption,comment}
\usepackage{mathtools}
\usepackage{amsthm}
\usepackage{todonotes}
\usepackage{mathrsfs,amssymb} 
\usepackage{graphicx,tikz,pgf,tkz-graph}
\usetikzlibrary{arrows,shapes}
\usetikzlibrary{decorations.pathreplacing}
\usepackage[shortlabels]{enumitem}

\newtheorem{theorem}{Theorem}

\newtheorem{remark}[theorem]{Remark}
\newtheorem{claim}[theorem]{Claim}

\newtheorem{proposition}[theorem]{Proposition}
\newtheorem{lemma}[theorem]{Lemma}

\newtheorem{conjecture}[theorem]{Conjecture}

\newtheorem{problem}[theorem]{Problem}

\newtheorem{question}[theorem]{Question}
\theoremstyle{definition}
\newtheorem*{definition*}{Definition}

\newtheorem*{hallstheorem1}{Hall's marriage theorem}

\newcommand*{\myproofname}{Proof}
\newenvironment{claimproof}[1][\myproofname]{\begin{proof}[#1]}{\end{proof}}

\DeclareMathOperator{\s}{s}

\let\le\leqslant
\let\ge\geqslant
\let\leq\leqslant

\let\geq\geqslant

\usepackage{bookmark}

\newcommand{\abs}[1]{\left|#1\right|}

\renewcommand*{\phi}{\varphi}
\renewcommand*{\emptyset}{\varnothing}

\renewcommand{\Pr}{\mathop\mathbb{P}\nolimits}
\newcommand*{\EE}{\mathop\mathbb{E}\nolimits}

\newcommand*{\sH}{\mathscr{H}}

\renewcommand{\a}{\alpha}
\newcommand{\binsq}{\mathbin\square}

\newcounter{row}
\newcounter{col}

\title{
List packing number of bounded degree graphs
}

\author{
	Stijn Cambie%
 \thanks{Extremal Combinatorics and Probability Group (ECOPRO), Institute for Basic Science (IBS), Daejeon, South Korea.
 Supported by IBS-R029-C4. Email: \protect\href{mailto:stijn.cambie@hotmail.com}{\protect\nolinkurl{stijn.cambie@hotmail.com}}.}	
	\and
	Wouter Cames van Batenburg%
	\thanks{Delft Institute of Applied Mathematics, Delft University of Technology, Netherlands.
		Email: \protect\href{mailto:w.p.s.camesvanbatenburg@tudelft.nl}{\protect\nolinkurl{w.p.s.camesvanbatenburg@tudelft.nl}}.}
	\and
	Ewan Davies%
	\thanks{Department of Computer Science, Colorado State University, Fort Collins, USA.
		Email: \protect\href{mailto:research@ewandavies.org}{\protect\nolinkurl{research@ewandavies.org}}.}
	\and
	Ross J. Kang%
	\thanks{Korteweg--de Vries Institute for Mathematics, University of Amsterdam, Netherlands.
		 Supported by a Vidi grant (639.032.614) of the Netherlands Organisation for Scientific Research (NWO). Email: \protect\href{mailto:r.kang@uva.nl}{\protect\nolinkurl{r.kang@uva.nl}}.}
}

\date{\today}

\begin{document}

\maketitle

\begin{abstract}
We investigate the list packing number of a graph, the least $k$ such that there are always $k$ disjoint proper list-colourings whenever we have lists all of size $k$ associated to the vertices. 
We are curious how the behaviour of the list packing number contrasts with that of the list chromatic number, particularly in the context of bounded degree graphs.
The main question we pursue is whether every graph with maximum degree $\Delta$ has list packing number at most $\Delta+1$.
Our results highlight the subtleties of list packing and the barriers to, for example, pursuing a Brooks'-type theorem for the list packing number.
\end{abstract}

\section{Introduction}

 List colouring is a well-known generalisation of graph colouring in which we wish to find a proper colouring of a graph, but an adversary supplies a list of colours for each vertex and we must choose a colour for each vertex from its list.
 The importance of this notion is its flexible role in inductive approaches, such as, for example, iterative random colouring procedures~\cite{MolloyReedbook}.
 Although there is already a rich collection of prominent challenges in list colouring (which relate well to algebraic, probabilistic, extremal, structural topics in graph theory), here we have set ourselves an even more difficult task of juggling multiple list-colourings simultaneously. 
 We initiated the study of this topic in~\cite{CCDK21}.
 
	Formally, a list-assignment $L$ of a graph $G$ is a function $L:V(G)\to 2^{\mathbb{N}}$, and an $L$-colouring of $G$ is a colouring $c \colon V(G)\to \mathbb{N}$ such that for every vertex $v$, $c(v)\in L(v)$. 
	As is standard, the $L$-colourings we consider in this work are usually assumed to be proper, that is, $c(v)\ne c(v')$ for any edge $vv'$.
	We call a list-assignment with $|L(v)|=k$ for all $v$ a $k$-list-assignment.
	Recall that the list chromatic number of $G$, denoted $\chi_\ell(G)$, is the smallest $k$ such that for every $k$-list-assignment $L$, $G$ admits an $L$-colouring. 
	Our work is motivated by the idea to look for multiple, disjoint list-colourings. 
	In particular, given a $k$-list-assignment $L$ of $G$ we call a collection of $k$ pairwise-disjoint $L$-colourings an $L$-packing of size $k$, or less specifically a list-packing.
	The \emph{list packing number} $\chi^\star_\ell(G)$ of $G$ is the least $k$ such that $G$ admits an $L$-packing of size $k$ for any $k$-list-assignment $L$ of $G$.
 Clearly, $\chi^\star_\ell(G)\ge \chi_\ell(G)$ always, and note how this implies the existence of a list-packing for any $k$-list-assignment $L$ where $k\ge\chi^\star_\ell(G)$ (by iteratively extracting $L$-colourings).

 The definitions above can easily be extended to a more general setup known as correspondence colouring. We defer the details of this more technical concept, as here it suffices to understand that replacing the notion of list colouring with correspondence colouring in the definitions above yields \emph{correspondence packing} and the \emph{correspondence packing number} we denote $\chi_c^\star(G)$. It holds that $\chi^\star_c(G) \ge \chi^\star_\ell(G)$ always.

The list packing number is in a natural progression from the list chromatic number, in a similar way to how the chromatic number relates to the independence number; we remark more on this in Subsection~\ref{sub:defns}.
 Indeed, list and correspondence packing were anticipated in earlier works---see, e.g.~\cite{AFH96,Cat80,KuOs09,Mac21,Yus21}---and it is thus surprising that our earlier work~\cite{CCDK21} was the first to systematically explore the notion, pointing to its potential throughout the landscape of (chromatic) graph theory.

 In our previous paper~\cite{CCDK21}, we pursued amongst other things packing versions of list and correspondence chromatic number bounds in terms of the number of vertices or degeneracy of the underlying graph.
 
 Considering the often deep relationship between the chromatic number and the independence number, the question directly comes to mind: truly how much ``worse'' is list packing in comparison to list colouring? In one interpretation of this question, what we audaciously named the \emph{List Packing Conjecture} in~\cite{CCDK21}, we proposed the list packing number might always be within some fixed constant factor of the list chromatic number. This mystery continues to fascinate us. 

 Although there are many tempting directions, more of which we mention further on, here we have restricted our attention to two of the most basic settings. We characterised graphs of list or correspondence packing number $2$. And we have made progress for bounded degree graphs, especially for those of small degree. 
 In this study, an approximate form of the packing numbers has naturally arisen.
 These settings have helped us uncover some interesting differences between finding a single list-colouring and a full packing of them, and have served as proving ground for more general intuition.

 The early, essential work of Erd\H{o}s, Rubin and Taylor on list colouring~\cite{ERT80} gives compass for results we may pursue for list and correspondence packing. In particular, Erd\H{o}s {\em et al.}~characterised the graphs of list chromatic number $2$ using the so-called {\em theta graphs}. Here we obtain a characterisation for packing, and find that the corresponding graph class is, naturally, smaller and simpler.
\begin{theorem}\label{thm:2packable}
 A graph $G$ has list packing number 2 if and only if it is a forest with at least one edge. The same holds for correspondence packing number.
 \end{theorem}

\noindent
Based on this, the next graphs to consider are of course cycles. 
We denote by $C_n$ a cycle on $n$ vertices.
Here we find that there is no dependence on parity of $n$, unlike for the respective colouring definitions.
	
\begin{theorem}\label{thm:cycles}
$\chi_\ell^\star(C_n)=3$ and $\chi_c^\star(C_n) = 4$.
\end{theorem}

These results give basis for the more difficult problem of list and correspondence packing in graphs of bounded maximum degree.
Before discussing our results in this setting, we present a conjecture which should serve as a focal point for future study.
These are prospective upper bounds for the list and correspondence packing numbers in terms of the maximum degree $\Delta(G)$ of a graph $G$.

	\begin{conjecture}\label{conj:chil*Delta}\label{conj:chic*Delta}
		For any graph $G$,\hfill
 \begin{enumerate}[(i)]
 \item\label{itm:chil*Delta}$\chi_{\ell}^\star(G) \le \Delta(G) +1$; and
 \item\label{itm:chic*Delta}$\chi_{c}^\star(G) \le 2 \left\lceil \frac{\Delta(G)+1}{2} \right \rceil $.
	\end{enumerate}
 \end{conjecture}

\noindent
 Either bound would be sharp for every choice of $\Delta(G)\ge 2$ by considering the complete graph. 
 For this, note that $\chi_{c}^\star(K_n) \ge \chi_{\ell}^\star(K_n) =n = \Delta(K_n)+1$ (the first equality was proved in~\cite{CCDK21}) and
 a construction of Catlin~\cite{Cat80} (see e.g.~\cite{Yus21}) shows that $\chi^\star_c(K_n) = n+1=\Delta(K_n)+2$ if $n\ge3$ is odd.
 In~\cite{CCDK21}, we showed bounds within about a factor $2$ of the conjectured bounds (though we slightly improve on that earlier result below). In the special case of $K_n$, Yuster~\cite{Yus21} proved a bound within a factor $1.78$ for all sufficiently large $n$. We note that Conjecture~\ref{conj:chic*Delta} is a broad generalisation of Conjecture 1.1 in~\cite{Yus21}, which in turn is related to the so-called ``modified Fischer's conjecture'' (see~\cite{KuOs09}).
 
 As evidence towards our conjecture, we have completely resolved a few of the first cases.
 \begin{theorem}\label{thm:smalldegrees}\hfill
 \begin{enumerate}[(i)]
 \item\label{itm:smalldegreesl}
 Conjecture~\ref{conj:chil*Delta}\ref{itm:chil*Delta} holds for $\Delta(G) = 2$ and $\Delta(G)=3$.
 \item\label{itm:smalldegreesc}
 Conjecture~\ref{conj:chic*Delta}\ref{itm:chic*Delta} holds for $\Delta(G) = 2$, $\Delta(G)=3$, and $\Delta(G)=4$.
 \end{enumerate}
 \end{theorem}

 \noindent
 We remark that this last case implies that $\chi_c^\star(K_5)=6$, which is a small step towards Conjecture~1.1 in~\cite{Yus21}.
 It also implies some natural subcases for the so-called {\em Strong Colouring Conjecture} (see~\cite{ABZ07}).
 Partly with computer assistance, we establish Theorem~\ref{thm:smalldegrees} in separate pieces; namely, it follows by combining Theorems~\ref{thm:chil*_Delta2}--\ref{thm:chic*_Delta3}, and~\ref{thm:maxdegree_improvement} below.

 As alluded to, we previously~\cite{CCDK21} established a bound about twice the conjectured optimal. In fact, we proved the following upper bound in terms of the degeneracy $\delta^\star(G)$ of $G$.

\begin{theorem}[{\cite[Thms.~3 and~9, Prop.~24]{CCDK21}}]\label{thm:degen}
For any graph $G$, $\chi_\ell^\star(G) \le \chi_c^\star(G) \le 2\delta^\star(G)$. The second inequality can be tight.
\end{theorem}

\noindent
 This implies an upper bound of $2\Delta(G)$. It has proven frustratingly difficult to significantly improve on this in general. While the following result gives an improvement that is modest, particularly for large $\Delta(G)$, we note the bound is best possible for $\Delta(G)=4$ (see Theorem~\ref{thm:smalldegrees}\ref{itm:smalldegreesc} above).

\begin{theorem}\label{thm:maxdegree_improvement}
For any graph $G$ with $\Delta(G) \ge 4$,
\(\chi^\star_\ell(G) \leq \chi^\star_c(G) \leq 2\Delta(G)-2.\)
\end{theorem}

 For larger $\Delta(G)$, although the above bound is still around a factor $2$ larger than we conjectured, it turns out that we can still find reasonable support for Conjecture~\ref{conj:chil*Delta} when we consider a mild relaxation of the list and correspondence packing numbers. We later give more formal definitions of the fractional list packing number $\chi_\ell^\bullet(G)$ and the fractional correspondence packing number $\chi_c^\bullet(G)$ of $G$. They are derived naturally from the idea that instead of integral packing, we could allow fractional weightings of the list- or correspondence-colourings of $G$.
 The following thus constitutes the confirmation of a fractional relaxation of Conjecture~\ref{conj:chil*Delta}.
 
\begin{theorem}\label{thm:fractionalgreedy}
 For any graph $G$, $\chi_\ell^\bullet(G)\le \chi_c^\bullet(G) \le \Delta(G)+1$.
\end{theorem}

\noindent
These fractional parameters may be of interest in their own right. Here and in Subsection~\ref{sub:contrast} we discuss some of their basic properties in comparison to their integral counterparts.

The following fractional form of our List Packing Conjecture as well as its correspondence analogue are worth further study.
\begin{conjecture}\label{conj:frac}\hfill
\begin{enumerate}[(i)]
 \item There exists $C>0$ such that $\chi_\ell^\bullet(G) \leq C \cdot \chi_{\ell}(G)$ for any graph $G$.
 \item There exists $C>0$ such that $\chi_c^\bullet(G) \leq C \cdot \chi_c(G)$ for any graph $G$.
\end{enumerate}
\end{conjecture}

\noindent
A challenge of dealing with list packing is how we must confront our usual intuition from colouring. In treating the fractional relaxation, we hoped to claw back some of that intuition, which we succeeded in doing in part in Theorem~\ref{thm:fractionalgreedy}. Here are some obstacles and subtleties we must further account for.

\begin{theorem}\label{thm:subtleties}\hfill
\begin{enumerate}[(i)]
 \item\label{itm:subtlety1} For each $d\ge 2$, there is a graph $G$ satisfying $\delta^\star(G)=d$ and $\chi_\ell^\bullet(G) \ge d+2$.
 \item\label{itm:subtlety2} For each $k\ge 2$, there is a graph $G$ satisfying $\chi_c^\bullet(G) = k+1$ and $\chi_c^\star(G) = 2k$.
 \item\label{itm:subtlety3} There is a connected $3$-regular graph $G \ne K_4$ such that $\chi_\ell^\star(G) \ge \chi_\ell^\bullet(G) = 4$.
\end{enumerate}
\end{theorem}

\noindent
Theorem~\ref{thm:subtleties}\ref{itm:subtlety1} shows it impossible to improve Theorem~\ref{thm:fractionalgreedy} by replacing maximum degree by degeneracy,
while Theorem~\ref{thm:subtleties}\ref{itm:subtlety3} shows how even a fractional relaxation of a list packing analogue must differ from the usual Brooks' theorem.
We prove items~\ref{itm:subtlety1} and~\ref{itm:subtlety2} in Section~\ref{sec:fractional}, and item~\ref{itm:subtlety3} in Subsection~\ref{sec:brooks}.

We conclude this introductory section by confronting yet another tempting intuition. Isn't it easy to construct a list-packing once we have sufficiently many list-colourings? In Subsection~\ref{subsec:linegraphs}, we show how this line of thought needs something extra. For every integer $n\geq 2$, we exhibit a graph $G$ on $n^2$ vertices together with an $n$-list-assignment $L$ such that $G$ has
\begin{itemize}
 \item list chromatic number $n$,
 \item $n^{n^2}$ not-necessarily-proper $L$-colourings, and
 \item $n^{n^2(1+o(1))}$ proper $L$-colourings,
\end{itemize}
and yet it does not admit an $L$-packing.

\subsection{Notation, definitions, preliminaries}\label{sub:defns}

The reader will have noticed that we use a variety of standard graph theoretic notation including $\Delta(G)$ for the maximum degree, $\omega(G)$ for the clique number, $\delta^\star(G)$ for the degeneracy, and $\chi(G)$ for the chromatic number of a graph $G$.

Here now are some of the more formal notation and definitions most directly related to our list packing parameters.
Let $G$ and $H$ be graphs. 
A pair $\sH=(L,H)$ is a \emph{correspondence-cover} of a graph $G$ if the graph $H$ and mapping $L:V(G)\to 2^{V(H)}$ satisfy that
\begin{enumerate}[(i)]
 \item $L$ induces a partition of $V(H)$,
 \item the bipartite subgraph of $H$ induced between $L(u)$ and $L(v)$ is empty whenever $uv\notin E(G)$,
 \item\label{itm:corrdef} the bipartite subgraph of $H$ induced between $L(u)$ and $L(v)$ is a matching whenever $uv\in E(G)$, 
 \item the subgraph of $H$ induced by $L(v)$ is a clique for each $v\in V(G)$.
\end{enumerate}
It can be convenient to drop $\sH$ and $L$ from the notation, saying for example that some property of the correspondence-cover holds if it holds for $H$ as a graph.
A correspondence-cover is \emph{$k$-fold} if $|L(v)|=k$ for each vertex $v$ of $G$.

Note that a list-assignment $L$ of $G$ naturally gives rise to a correspondence-cover $(\tilde L, H)$ of $G$ by setting $\tilde L(v) = x_v$ for each $x\in L(v)$, and forming $H$ on these vertices by putting in the necessary cliques and adding edges of the form $x_ux_v$ for each colour $x\in\bigcup_{v\in V(G)}L(v)$ and edge $uv\in E(G)$. 
We call a correspondence-cover that arises from a list-assignment in this way a \emph{list-cover} of $G$.

We remark that in the literature, the concept of a cover is more general than a correspondence-cover, namely by the omission of condition~\ref{itm:corrdef}. Moreover, it can be defined in various ways, where in particular the cliques on the lists can be omitted.
In this paper, we embrace the inclusion of these cliques for convenience.
In this way, the definition of \emph{list packing number} $\chi_\ell^\star(G)$ given above is equivalent to the least $k$ such that for every $k$-list-assignment $L$, the associated cover $(\tilde L, H)$ has chromatic number $k$. 
Under this equivalence, each colour class in a proper $k$-colouring of $H$ is precisely an $L$-colouring of $G$.

It is now straightforward to define correspondence colouring and packing. 
Given a correspondence-cover $\sH=(L,H)$ of a graph $G$, we say that an $\sH$-colouring of $G$ is an independent set of size $|V(G)|$ in $H$ and an $\sH$-packing of $G$ of size $k$ is a partition of $H$ into $k$ $\sH$-colourings of $G$.
The \emph{correspondence chromatic number} $\chi_c(G)$ of $G$ is the least $k$ such that 
every $k$-fold correspondence-cover $\sH$ of $G$
has an independent set of size $|V(G)|$.
Similarly, the \emph{correspondence packing number} $\chi_c^\star(G)$ is the least $k$ such that  
every $k$-fold correspondence-cover $\sH$ of $G$ has chromatic number $k$ (i.e.\ every $k$-fold correspondence-cover $\sH$ of $G$ admits an $\sH$-packing).

One can interpret a list- or correspondence-packing of $G$ (of size $k$) as an assignment of $\{0,1\}$-weights to every possible independent set of size $|V(G)|$ in the corresponding cover of $G$ such that each cover vertex is assigned weight exactly once by some independent set containing it (and exactly $k$ independent sets are assigned nonzero weight).
Then, as is common in combinatorics and optimisation, one can naturally ``fractionally'' relax the constraint that the weights be integral (so they may take values in the interval $[0,1]$), while demanding that the total weight assigned to a cover vertex is $1$ (and that the sum of the weights of the independent sets is exactly $k$).
We find it particularly interesting to consider fractional relaxations of the packing numbers in our investigation of whether known results for $\chi_\ell$ and $\chi_c$ extend to the packing variants $\chi_\ell^\star$ and $\chi_c^\star$.
The \emph{fractional list packing number} of $G$, denoted $\chi_\ell^\bullet(G)$, is the least integer $k$ such that every $k$-fold list-cover of $G$ has fractional chromatic number $k$.
Similarly, the \emph{fractional correspondence packing number} of $G$, denoted $\chi_c^\bullet(G)$, is the least integer $k$ such that every $k$-fold correspondence-cover of $G$ has fractional chromatic number $k$.

Note that the fractional packing numbers can only take on integer values; their fractional character lies in the manner in which weights may be distributed over the independent sets, but is not reflected in the sum of the weights. We remark that there are other viable notions of fractional packing numbers, but we found it most natural to introduce the above.

Since every list-cover is a correspondence-cover we have
\begin{align*}
\chi_{\ell}(G) &\leq \chi_c(G),&
\chi_{\ell}^{\bullet}(G) &\leq \chi_c^{\bullet}(G),&\chi_{\ell}^{\star}(G) &\leq \chi_c^{\star}(G);
\end{align*}
and since the fractional chromatic number is a lower bound for chromatic number, the following inequalities also immediately follow from the above definitions:
\begin{align*}
 \chi_{\ell}(G) &\leq \chi_{\ell}^{\bullet}(G) \leq \chi_{\ell}^{\star}(G),& \chi_c(G) &\leq \chi_{c}^{\bullet}(G) \leq \chi_c^{\star}(G).
\end{align*}

We prove at the end of Section~\ref{sec:fractional} that these last inequalities can be strict.

\label{sub:contrast}

\begin{proposition}\label{prop:ineqs}
 For each of the following four inequalities, there is some graph $G$ satisfying it:
 \begin{align*}
 \chi_\ell(G)&<\chi_\ell^\bullet(G),&
 \chi_\ell^\bullet(G)&<\chi_\ell^\star(G),&
 \chi_c(G)&<\chi_c^\bullet(G),&
 \chi_c^\bullet(G)&<\chi_c^\star(G).
 \end{align*}
\end{proposition}

The next result shows that a fractional packing is a stronger notion than the existence of a colouring extending every possible assignment of a colour to a vertex, and that fractional packings must be probability distributions over full list-colourings of the graph, i.e.~they cannot assign positive weight to non-maximum independent sets in the cover graph.

\begin{proposition}\label{prop:fracfacts}
 Let $G$ be a graph and $k\ge \chi_\ell^\bullet(G)$. Then for any $k$-list-assignment $L$ of $G$, for any $v\in V(G)$ and $x\in L(v)$, there is a proper $L$-colouring $c$ of $G$ with $c(v)=x$. 

 Similarly, for any $k$-fold list-cover $H$ of $G$ and fractional colouring $c$ of $H$ of weight $k$, $c$ assigns positive weight only to maximum independent sets of $H$, which are of size $|V(G)|$.

 These statements also hold, \emph{mutatis mutandis} for $\chi_c^\bullet$ and correspondence-covers.
\end{proposition}
\begin{proof}
 The first statement follows from the second; the existence of a fractional colouring supported only on maximum independent sets implies that every vertex is contained in a maximum independent set.

 The second statement is a simple consequence of the bound $\chi_f(H) \ge N/\alpha$ that holds for any $N$-vertex graph $H$ with independence number $\alpha$.
 Every $k$-fold list-cover $H$ of $G$ has $|V(H)|=k|V(G)|$ and $\alpha(H)\le |V(G)|$, with equality when $k\ge \chi_\ell(G)$. 
 So if $\chi_f(H) = k = |V(H)|/\alpha(H)$, we have equality.
 Now, if $c$ is a fractional colouring of $H$ of weight $k$, which we interpret as a probability distribution on independent sets $I$ of $H$ such that $\Pr_c(x\in I)\ge1/k$ for each $x\in V(H)$, we have 
 \[ |V(G)| = \alpha(H) \ge \EE_c|I| = \sum_{v\in V(H)}\Pr_c(v\in I) \ge |V(H)|/k = |V(G)|, \]
 and hence every $I$ with positive probability has size $|V(G)|$.

 The same proofs work in exactly the same way for correspondence packing.
\end{proof}

By the previous, it is not hard to see that the following is an alternative, equivalent definition for the fractional list packing number.

\begin{definition*}
 Given a $k$-list-assignment $L$ of $G$, a \emph{fractional $L$-packing of $G$} is (for some $m \in \mathbb Z^+$) a collection of $mk$ (not necessarily distinct) proper $L$-colourings $c_1,\dots,c_{mk}$ of $G$, such that for every $v \in V$ and $c \in L(v)$ there are $m$ values $i \in [mk]$ for which $c_i(v)=c.$
 The smallest value of $k$ for which a fractional $L$-packing of $G$ exists for every $k$-list-assignment $L$ of $G$, is the fractional list packing number $\chi^\bullet_\ell(G).$ 
\end{definition*}

When working with explicit covers, we will often consider permutations of sets such as $\{1,2,3,4\}$ and $\{1_x,2_x,3_x\}$ endowed with a natural order that we assume is clear. 
It can be convenient to omit the subscripts, and usually we write a permutation of the set as an ordered sequence of comma-separated values, such as $f=(2,1,3)$ for the permutation $f$ with $f(1)=2$, $f(2)=1$, $f(3)=3$.
We do use standard cycle notation for transpositions, e.g.\ $(1\, 2)$ is a permutation of any ground set containing $\{1,2\}$ that swaps $1$ and $2$. 
The ground set will be clear from context.

We will use the following formulation of Hall's marriage theorem~\cite{Hal35}.

\begin{hallstheorem1}[\cite{Hal35}]\em
Given a family $\mathcal{F}$ of finite subsets of some ground set $X$, where the subsets are counted with multiplicity, suppose $\mathcal{F}$ satisfies the \emph{marriage condition}, that is that for each subfamily $\mathcal{F}' \subseteq \mathcal{F}$
\begin{align*}
|\mathcal{F}'| \le \left|\bigcup_{A\in \mathcal{F}'} A\right|.
\end{align*}
Then there is an injective function $f: \mathcal{F}\to X$ such that $f(A)$ is an element of the set $A$ for every $A\in \mathcal{F}$, that is, the image $f(\mathcal{F})$ is a \emph{system of distinct representatives} of $\mathcal{F}$.
\end{hallstheorem1}

\noindent
This can also be stated in the terminology of matchings in bipartite graphs.

\begin{theorem}[Hall's marriage theorem, graph theoretic formulation]\label{thm:Hall_graphversion}
If $G=(A \cup B,E)$ is a bipartite graph for which $\lvert N(A_1) \rvert \ge \lvert A_1 \rvert$ for every $A_1 \subseteq A$, then $G$ has a (maximum) matching of size $\lvert A \rvert$.
\end{theorem}

\subsection{Outline of the paper}\label{sub:outline}
	
The proofs for the upper bounds of the packing numbers of maximum degree $2,3$ and $4$ are given in the corresponding Sections~\ref{sec:paths+cycles},~\ref{sec:subcubic} and~\ref{sec:maxdegree}. In Section~\ref{sec:fractional} we give some results on fractional packing numbers.
Finally, in Section~\ref{sec:concrem} we conclude with observations on list packing of edge-colourings, possible variants of Brooks' theorem, and comments on algorithmic aspects.

\section{Paths and cycles}\label{sec:paths+cycles}

In this section, we determine the list and correspondence packing numbers for graphs with maximum degree $2$, the connected examples of which are paths and cycles. 
The main work is for cycles. 
For the list packing number, we prove that for a particular edge $uv$, a partial list-packing of $C_n \setminus uv$ can be extended simultaneously to both $u$ and $v$. 
For the correspondence packing number, we prove that it is strictly larger than $3$ by giving a construction. 

\begin{theorem}\label{thm:chil*_Delta2}
 For $n\ge 2$, the list packing number of the $n$-vertex path $P_n$ is $2$.
	For $n\ge 3$, the list packing number of the cycle $C_n$ is $3$.
\end{theorem}

\begin{proof}
The first statement follows from Theorem~\ref{thm:degen} since $2=\chi(P_n) \le \chi_{\ell}^\star(P_n) \le 2 \delta^\star(P_n)=2$.

For cycles, we first prove $\chi_\ell^\star(C_n)\ge 3$.
When $n$ is odd we have $\chi_{\ell}^\star(C_n) \ge \chi(C_n)=3$. 
When $n$ is even, we give a $2$-list-assignment $L$ of $C_n$ which does not admit an $L$-packing. Let the lists of the consecutive vertices $v_1$ up to $v_n$ of the cycle be $L_1, L_2, \ldots, L_n$ with
$L_i =\{1,2\}$ for $1 \le i \le n-2$, $L_{n-1}=\{1,3\}$ and $L_{n}=\{2,3\}$.
To rule out an $L$-packing, it is sufficient to observe that no proper $L$-colouring $c$ can satisfy $c(v_1)=2$.
An $L$-colouring $c$ with $c(v_1)=2$ must have $c(v_i)=2$ for every odd $i$ such that $1 \le i \le n-3$, and $c(v_i)=1$ for every even $i$ with $1\le i \le n-2$.
Continuing the colouring, $c(v_{n-1})$ must be $3$, and going backwards from $v_1$ we see that $c(v_n)$ must also be $3$, so $c$ cannot be proper.
As there does not exist a list-packing in this case, we have $\chi_{\ell}^\star(C_n)\ge3$. 
Alternatively, one can construct the cover $(\hat L,H)$ of $C_n$ using these lists and observe that it contains an odd cycle. Thus, the 2-fold list-cover has chromatic number strictly greater than $2$.

Now we prove the upper bound, that for every list-assignment $L$ of $C_n$ with lists of size 3 there exists an $L$-packing.
If all lists are the same, this is clear since we can pick an arbitrary proper colouring and take two translates of that one.
Then the remaining case involves adjacent vertices $u$ and $v$ such that $L(u) \ne L(v)$, and hence $\lvert L(u) \cap L(v) \rvert \le 2$. The hardest case to prove is when the intersection has size exactly $2$.
Without loss of generality, in this case we can take $L(u) =\{1,2,3\}$ and $L(v)=\{1,2,4\}$.
First, take an arbitrary $L$-packing $\vec{c} = (c_1, c_2, c_3)$ of $C_n \setminus \{u,v\}$ (which is isomorphic to a path $P_{n-2}$),
The worst case is that the neighbour $w$ of $u$ in $C_n\setminus\{u,v\}$ has the same list $L(w)=\{1,2,3\}$ as $u$, because if this is not the case there will simply be more options for extending the $L$-packing to $u$. 
Then without loss of generality we can assume that $c$ gives $w$ the colours $(c_1(w),c_2(w),c_3(w))=(1,2,3)$ in order, and hence $\vec c(u) =(3,1,2)$ and $\vec c(u) = (2,3,1)$ are both valid extensions of $c$ to $u$.
Similarly, there will be at least two possible choices for the extension of $c$ to $v$. 
We are done if out of the (at least) four possible extensions of $c$ to both $u$ and $v$, one of them is valid on the edge $uv$.
It is easy to see that among all permutations of $\{1,2,4\}$, only the choice $\vec c(v) = (2,1,4)$ would yield a situation in which case the extension to $u$ is impossible. So there is always a choice for $\vec c(v)$ such that the $L$-packing can be completed.
\end{proof}

\begin{theorem}\label{thm:chic*_Delta2}
	For $n\ge2$, the correspondence packing number of the $n$-vertex path $P_n$ is $2$.
	For $n\ge3$, the correspondence packing number of the cycle $C_n$ is $4$.
\end{theorem}

\begin{proof}
 	The first statement is true since $2=\chi(P_n) \le \chi_{c}^\star(P_n) \le 2 \delta^\star(P_n)=2$.
	The upper bound $\chi_{c}^\star(C_n)\le 2 \delta^\star(C_n)=4$ follows from Theorem~\ref{thm:degen} as well.
	It remains to give, for every $n\ge 3$, a $3$-fold correspondence-cover $\sH=(L,H)$ of $C_n$ for which no $\sH$-packing exists.

 Let $xy$ be an arbitrary edge of $C_n$, and for each $v\in C_n$, let $L(v)=\{1_v,2_v,3_v\}$.
 Form the cover $H$ by connecting $i_u$ and $i_v$ for every edge $uv\in E(C_n)\setminus \{xy\}$. 
 Between $L(x)$ and $L(y)$, we connect $2_x$ to $3_y$ and $3_x$ to $2_y$, as well as $1_x$ to $1_y$. This is presented for $C_4$ and $C_5$ in Figure~\ref{fig:CorrespondenceCoverCycles}.
	Assume there is an $\sH$-packing, i.e.\ there is a vector $\vec c = (c_1,c_2,c_3)$ such that $c_i$ is an independent transversal with $c_i(v) \in L(v)$, and such that $c_i(v) \ne c_j(v)$ for $i \ne j$.
	For every $v$, we have that $\vec c(v)=(c_1(v),c_2(v),c_3(v))$ is a permutation of $(1_v,2_v,3_v)$.
	Furthermore, for any edge $uv$ other than $xy$, when we drop the subscripts and consider $\vec c(u)$ and $\vec c(v)$ as permutations of $\{1,2,3\}$, they do not have an index mapped to the same value.
 That is, $\vec c(v)\vec c(u)^{-1}$ is a derangement. In particular, $\vec c(v)\vec c(u)^{-1}$ is an even permutation because the only derangements of $\{1,2,3\}$ are the even permutations $(2,3,1)$ and $(3,1,2)$. 
 On the other hand, for the edge $xy$ we must have that $\vec c(x)\vec c(y)^{-1}$ is an odd permutation.
	This leads to a contradiction as follows. Permuting the labels if necessary, we may assume that $\vec c(x) = (1_{x}, 2_{x}, 3_{x})$ is the identity permutation. The above argument shows that going around the cycle in order, starting at $x$ and going away from $y$, each $\vec c(v)$ must be an even permutation, including $\vec c(y)$. We now have a contradiction as to prevent the packing from containing colourings that make $xy$ monochromatic, we must have that $\vec c(y)$ is an odd permutation.
\end{proof}	

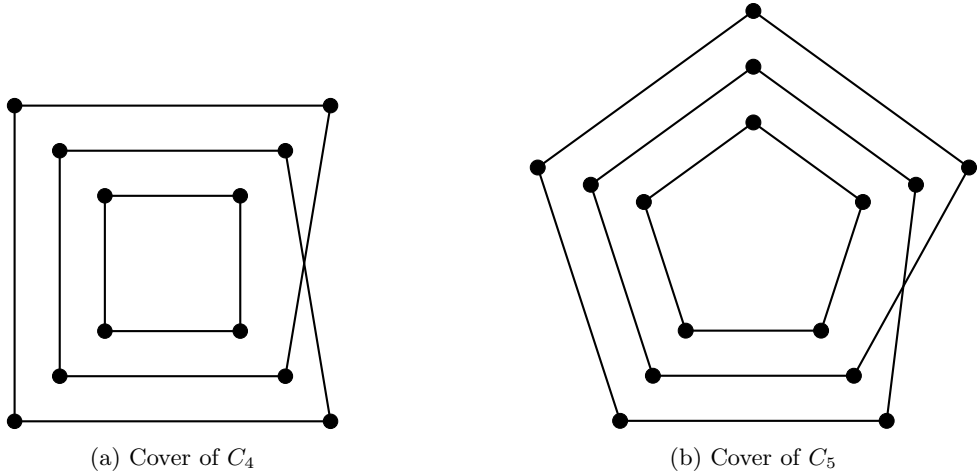
\begin{figure}[ht]
 \centering
 \subcaptionbox{Cover of $C_4$\label{fig:C4cover}}[.4\textwidth]{
 \begin{tikzpicture}[scale=0.6]

 \draw [thick] (-7, 1) -- (-7,-6);
 \draw [thick] (-7,-6) -- (0,-6);
 \draw [thick] ( 0,-6) -- (-1,0);
 \draw [thick] ( 0, 1) -- (-1,-5);
 \draw [thick] (-2,-1)-- (-2,-4);
 \draw [thick] (-2,-4)-- (-5,-4);
 \draw [thick] (-5,-1)-- (-2,-1);
 \draw [thick] (-1, 0)-- (-6,0);
 \draw [thick] (-7, 1)-- (0,1);
 \draw [thick] (-6, 0)-- (-6,-5);
 \draw [thick] (-5,-1)-- (-5,-4);
 \draw [thick] (-1,-5)-- (-6,-5);
 
 \draw [fill] (-5,-4) circle (0.16);
 \draw [fill] (-2,-4) circle (0.16);
 \draw [fill] (-2,-1) circle (0.16);
 \draw [fill] (-5,-1) circle (0.16);
 \draw [fill] (-6,-5) circle (0.16);
 \draw [fill] (-1,-5) circle (0.16);
 \draw [fill] (-1, 0) circle (0.16);
 \draw [fill] (-6, 0) circle (0.16);
 \draw [fill] (-7,-6) circle (0.16);
 \draw [fill] ( 0,-6) circle (0.16);
 \draw [fill] ( 0, 1) circle (0.16);
 \draw [fill] (-7, 1) circle (0.16);
 \end{tikzpicture}}
 \qquad\qquad
 \subcaptionbox{Cover of $C_5$\label{fig:C5cover}}[.4\textwidth]{
 \begin{tikzpicture}[scale=0.3]
 \draw [thick] (23,10.177322521524923)-- (13.443616015410528,3.2342031427713476);
 \draw [thick] (23,7.705186566525345)-- (15.794757024580422,2.4702711202711374);
 \draw [thick] (23,5.233050611525759)-- (18.145898033750317,1.706339097770922);
 \draw [thick] (18.145898033750317,1.706339097770922)-- (20,-4);
 \draw [thick] (20,-4)-- (26,-4);
 \draw [thick] (27.854101966249686,1.70633909777092)-- (26,-4);
 \draw [thick] (23,5.233050611525759)-- (27.854101966249686,1.70633909777092);
 \draw [thick] (30.205242975419576,2.4702711202711347)-- (28.906170112021446,-8);
 \draw [thick] (27.453085056010725,-6)-- (32.55638398458947,3.234203142771342);
 \draw [thick] (32.55638398458947,3.234203142771342)-- (23,10.177322521524923);
 \draw [thick] (23,7.705186566525345)-- (30.205242975419576,2.4702711202711347);
 \draw [thick] (15.794757024580422,2.4702711202711374)-- (18.546914943989275,-6);
 \draw [thick] (18.546914943989275,-6)-- (27.453085056010725,-6);
 \draw [thick] (28.906170112021446,-8)-- (17.093829887978554,-8);
 \draw [thick] (13.443616015410528,3.2342031427713476)-- (17.093829887978554,-8);
 \draw [fill] (20,-4) circle (0.33);
 \draw [fill] (26,-4) circle (0.33);
 \draw [fill] (27.854101966249686,1.70633909777092) circle (0.33);
 \draw [fill] (23,5.233050611525759) circle (0.33);
 \draw [fill] (18.145898033750317,1.706339097770922) circle (0.33);
 \draw [fill] (18.546914943989275,-6) circle (0.33);
 \draw [fill] (17.093829887978554,-8) circle (0.33);
 \draw [fill] (27.453085056010725,-6) circle (0.33);
 \draw [fill] (28.906170112021446,-8) circle (0.33);
 \draw [fill] (30.205242975419576,2.4702711202711347) circle (0.33);
 \draw [fill] (23,7.705186566525345) circle (0.33);
 \draw [fill] (15.794757024580422,2.4702711202711374) circle (0.33);
 \draw [fill] (32.55638398458947,3.234203142771342) circle (0.33);
 \draw [fill] (23,10.177322521524923) circle (0.33);
 \draw [fill] (13.443616015410528,3.2342031427713476) circle (0.33);
 \end{tikzpicture}}
	\caption{Correspondence-covers of cycles used in the proof of Theorem~\ref{thm:chic*_Delta2}. For clarity, we omit the cliques on the sets $L(v)$.\label{fig:CorrespondenceCoverCycles}}
\end{figure}

Theorem~\ref{thm:cycles} on the packing numbers of cycles is a direct consequence of Theorems~\ref{thm:chil*_Delta2} and~\ref{thm:chic*_Delta2}. Using this we can also characterise the graphs with packing numbers $2$.

\begin{proof}[Proof of Theorem~\ref{thm:2packable}]
A forest with at least one edge is $1$-degenerate, so by~Theorem~\ref{thm:degen} it has correspondence packing number at most $2$. Since $\chi(K_2)=2$, the list and correspondence packing number must in fact be exactly $2$. Conversely, if a graph is not a forest then it contains a cycle which via Theorem~\ref{thm:cycles} forces the list packing number to exceed $2$.
\end{proof}

\noindent
We note in passing that Theorems~\ref{thm:2packable} and~\ref{thm:cycles} imply that no graph $G$ exists for which $\chi_c^\star(G)=3.$ In~\cite{CH23}, it is proved that $3$ is the only positive integer that cannot be attained by the correspondence packing number.

\section{Subcubic graphs}\label{sec:subcubic}

In this section, we prove that both the list and correspondence packing numbers of subcubic graphs are at most $4$.
Here, it is sufficient to consider cubic graphs because a connected subcubic graph which is not 3-regular has degeneracy at most $2$, in which case the result follows from Theorem~\ref{thm:degen}.
The idea of extending an $L$-packing of $G \setminus \{u,v\}$, as done in Section~\ref{sec:paths+cycles}, does not always work in 3-regular graphs.
As such, we need to do substantially more work.
First, we verify the case $K_4$ separately, by a computer search over all $4$-fold correspondence-covers of $K_4$.
In a 3-regular graph that is not $K_4$, we can take an edge $uv$ which does not belong to any triangle, and then consider partial packings of an even smaller subgraph than $G \setminus \{u,v\}$, as we must also take some care when packing certain neighbours of $u$ and $v$.
Finishing the proof requires an argument similar to the proof of Theorem~\ref{thm:chil*_Delta2}, where we show that there must be a valid extension to both $u$ and $v$ that is also valid for the edge $uv$, but there are so many cases to check that it is convenient to verify them with computer assistance.

\begin{theorem}\label{thm:chic*_Delta3}
	For every graph $G$ with maximum degree $\Delta(G)\le 3$, we have $\chi_{\ell}^\star(G) \le \chi_{c}^\star(G) \le 4$.
\end{theorem}

\begin{proof}
	It suffices to prove, for every graph $G$ with maximum degree $3$ and any $4$-fold correspondence-cover $\sH=(L,H)$ of $G$, there exists a correspondence-packing. Clearly, it is sufficient to check all $4$-fold correspondence-covers in which full matchings are taken between $L(u)$ and $L(v)$ for each edge $uv$ of $G$. 
 One only has to check connected graphs $G$, and by Theorem~\ref{thm:degen} we only have to consider cubic graphs. Briefly, in the case that $G$ has a vertex $v$ with degree at most $2$, one can extend a correspondence-packing on $G \setminus \{v\}$ by Hall's marriage theorem as shown in the proof of Theorem~\ref{thm:degen} given as~\cite[Thm.~9]{CCDK21}.

 For the complete graph $K_4$, it is known that $\chi_{c}^\star(K_4)=4$. 
 This can be checked by brute force, as we have done with Sage\footnote{\url{https://github.com/StijnCambie/ListPackII}, document chic(K4).py}, and Yuster did in~\cite[App.~A]{Yus21}.

 Let $G=(V,E)$ be an arbitrary connected cubic graph on $n$ vertices such that $G$ is not $K_4$.
	By the following claim, $G$ has an edge $uv$ which is not part of a triangle. 
	\begin{claim}\label{clm:edgenotinK3}
	 Every connected cubic graph $G$ which is not equal to $K_4$, has an edge $uv$ which is not part of a triangle.
	\end{claim}
	\begin{claimproof}
	 Assume not. Let $a\in G$ be a vertex with $3$ neighbours $b$, $c$ and $d$.
	 The edges $ab, ac$ and $ad$ all belong to a triangle. 
	 This implies that there are at least $2$ triangles containing $a$ and there cannot be $3$ triangles containing $a$ as otherwise $G[\{a,b,c,d\}]$ would be isomorphic to $K_4$.
	 Without loss of generality, assume that $abc$ and $acd$ are triangles, and the edge $bd$ is the only edge missing.
	 Let $e$ be the third neighbour of $d$.
	 Since $a$ and $c$ already have degree $3$, they are not neighbours of $e$.
	 As such, $d$ and $e$ have no common neighbours, implying that $de$ is an edge not belonging to a triangle.
	\end{claimproof}

 Using the claim, let $uv$ be an edge of $G$ not contained in a triangle, and let $\sH=(L,H)$ be a $4$-fold correspondence-cover of $G$. 
 We let $u_1$, $u_2$ and $v_1$, $v_2$ be the two neighbours of $u$ and $v$ respectively, different from $u$ and $v$ themselves.
	The vertices $X=\{u_1,u_2, v_1,v_2\}$ are distinct by the choice of the edge $uv$.
 Since the edges $F=\{uu_1, uu_2, uv, vv_1, vv_2\}$ form a tree on $X$, we can `untwist' the (full) matchings in $H$ covering $F$ without loss of generality. That is, 
 we can label $L(x)=\{1_x,2_x,3_x,4_x\}$ for each $x\in X$ such that for each $xy\in F$, the matching in $H$ between $L(x)$ and $L(y)$ is the ``identity'' connecting $i_x$ to $i_y$ for $1\le i\le 4$.
	Let $\vec{c} = (c_1, c_2, c_3, c_4)$ be a correspondence-packing for $G\setminus\{u,v,u_2,v_1,v_2\}$, which exists by Theorem~\ref{thm:degen}.
 We can assume without loss of generality that $\vec c(u_1)=(c_1(u_1),c_2(u_1),c_3(u_1),c_4(u_1))=(1_{u_1},2_{u_1},3_{u_1},4_{u_1})$, and it is sufficient to prove that this packing $\vec c$ can be extended to a correspondence-packing for $G$.
 We can drop the subscripts and consider the vectors $\vec c(x)$, which we must define for $x\in X\setminus\{u_1\}$, as permutations of $\{1,2,3,4\}$.

\begin{figure}[H]
 \centering
 \begin{tikzpicture}
 \foreach \x in {0,4.25}
 {
 \draw[fill] (\x,0) circle (0.1);
 \draw[fill] (\x-1,2) circle (0.1);
 \draw[fill] (\x+1,2) circle (0.1);
 \draw[thick] (\x+1,2)--(\x,0) -- (\x-1,2);
 }
 \draw[thick] (0,0)--(4.25,0);

 \foreach \x in {-1,1,3.25,5.25}{

 \draw[dotted] (\x-0.25,2.5)--(\x,2)--(\x+0.25,2.5);
 }

 \node at (-1.35,2) {$u_1$};
 \node at (1.35,2) {$u_2$};
 \node at (2.9,2) {$v_1$};
 \node at (5.6,2) {$v_2$};
 \node at (0,-0.3) {$u$};
 \node at (4.25,-0.3) {$v$};
 \node at (0,-0.7) {$\vec{c}(u)$};
 \node at (4.25,-0.7) {$\vec{c}(v)$};

 \node at (-2.7,1) {$\vec c(u_1)=\begin{pmatrix} 1_{u_1}\\ 2_{u_1}\\ 3_{u_1}\\ 4_{u_1}\\ \end{pmatrix}$};

 \node at (1.35, 1.6) {$\vec{c}(u_2)$};
 \node at (2.9, 1.6) {$\vec{c}(v_1)$};
 \node at (5.6, 1.6) {$\vec{c}(v_2)$};
 \end{tikzpicture}
 \caption{The local structure of $G$. Note that there could be edges amongst $\{u_1,u_2,v_1,v_2\}$ which we do not attempt to picture.}\label{fig:localpart}
\end{figure}
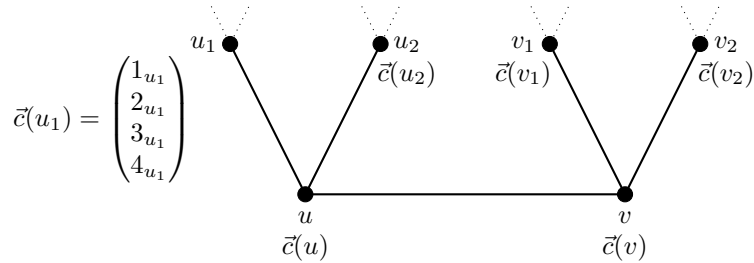

Starting from $\vec{c}$ (the partial packing for $G\setminus\{u,v,u_2,v_1,v_2\}$), we now do the following.
\begin{enumerate}
 \item Choose a permutation $\vec c(u_2)$ different from $(2, 3, 4, 1)$, $(2, 4, 1, 3)$, $(3, 1, 4, 2)$ and $(4, 1, 2, 3)$ such that $\vec c$ is a proper packing of $G\setminus\{u, v, v_1,v_2\}$. 
 Using a computer, we show that this can be done for any possible placement of the edges of $G\setminus\{u, v, v_1,v_2\}$ and choices of matchings covering the edges incident to $u_2$, using the fact that $u_2$ has at most two neighbours in $V\setminus\{u,v,v_1,v_2\}$. 
 
 \item Given the four special exclusions, there are 20 remaining permutations which $\vec c(u_2)$ could be. We show that these 20 choices can be partitioned as follows. 
 \begin{enumerate}[(i)]
 \item There are 10 `excellent' choices, for which the packing can be greedily extended to $v_1$, and then $v_2$ such that a valid choice of $(\vec c (u), \vec c (v))$ remains. 
 \item There are 8 `good' choices such that, provided exactly one problematic set $\{\vec c(v_1), \vec c(v_2)\}$ is avoided, a valid (at least one) choice of $(\vec c (u), \vec c (v) )$ remains.
 Avoiding one problematic set is always possible, since there are at least two available permutations when extending the partial packing to a vertex which has at most $2$ neighbours which are already coloured/packed.
 \item There are 2 `bad' choices such that there are 8 further problematic sets $\{\vec c(v_1), \vec c(v_2)\}$ that must be avoided. In this case it is not immediate that choices avoiding these problematic sets are possible, but we verify that the packing can be completed nonetheless.
 \end{enumerate}
\end{enumerate}

 Having given the idea of the steps to extend the partial list-packing, we now give the details why it works.
 Using a computer program\footnote{\url{https://github.com/StijnCambie/ListPackII}, document chic(Delta=3).py} we can list all possible choices (there are $112$ of them) for $(\vec c(u_2), \vec c(v_1), \vec c(v_2))$ for which the list-packing cannot be extended to both $u$ and $v$, i.e. to a full proper list-packing of $G$.
 In the code, this is marked by the comment ``In [1]''. Intuitively, since $112\ll(4!)^3$, there are only few problematic choices for $(\vec c(u_2), \vec c(v_1), \vec c(v_2))$ and we can avoid these, as we show next.
 
 No matter the colourings of the neighbours of $u_2$ different from $u$, we can choose $\vec c(u_2)$ different from $(2, 3, 4, 1)$, $(2, 4, 1, 3)$, $(3, 1, 4, 2)$ and $(4, 1, 2, 3)$.
 For this, we note that at most two neighbours of $u_2$ can be packed by $\vec c$ thus far.
 This is done at the comment ``In [4]'' in the code.
 Going through the $112$ bad triples at ``In [3]'', one concludes that there are $6$ ``bad'' choices for $\vec c(u_2)$ belonging to $16$ non-extendable triples, and $8$ ``good'' possibilities for $\vec c(u_2)$ to $2$ non-extendable triples.
 In the latter case, there is only one set $\{\vec c(v_1), \vec c(v_2)\} $ for which the extension was impossible (since switching the two gives the same obstruction). 
 In those cases one can always choose $\vec c(v_1)$ and $\vec c(v_2)$ such that they are not equal to such a bad set, since once the derangements of two neighbours of $v_2$ are chosen, there are at least two possible extensions for $v_2$ (see ``In [2]'').

 That is, in the good cases one can indeed complete the packing. 
 
 We can afford to reduce the $6$ ``bad'' cases to $2$, as there is enough flexibility to take $\vec c(u_2)$ different from $4$ of the $6$ bad choices.
 That is, up front we choose $\vec c(u_2)$ different from the bad choices $(2, 3, 4, 1)$, $(2, 4, 1, 3)$, $(3, 1, 4, 2)$ and $(4, 1, 2, 3)$.
 The remaining bad cases are when $\vec c(u_2)$ is equal to either $(3, 4, 2, 1)$ or $(4, 3, 1, 2)$.
 But in these two last cases, one can extend the partial (correspondence) colourings to $v_1$ and $v_2$, in such a way that $\vec c(v_1)$ and $\vec c(v_2)$ are not taken from the set $\{(1, 3, 2, 4), (3, 2, 1, 4), (4, 2, 3, 1), (1, 4, 3, 2)\}$ (see ``In [5]''). 
 This this case, we verify at ``In [6]'' that the triple of choices $(\vec c(u_2), \vec c(v_1), \vec c(v_2))$ does permit an extension of $\vec c$ to both $u$ and $v$ as required. 
 As such, we conclude that we always can choose $\{\vec c(x) \mid x \in \{u_1,u_2,v_1,v_2\} \}$ such that the correspondence-packing can be extended to $u$ and $v$ as well, i.e.\ we have a correspondence-packing for $G$.
\end{proof}

\section{Larger maximum degree}\label{sec:maxdegree}

In this section, we improve the upper bound $\chi_c^\star(G)\le2\Delta(G)$ (which follows from Theorem~\ref{thm:degen}) whenever $\Delta(G) \ge 4$. 
The method is a more careful analysis of Hall's marriage theorem, the main technique for proving Theorem~\ref{thm:degen}. For $\Delta(G)=4$, this results in a sharp upper bound for the correspondence packing number.

We start by stating some specific corollaries of Hall's marriage theorem.
The first version is just for completeness, but also indicates the nice structure of the counterexamples in the case when the conditions in Hall's marriage theorem are almost met.

 \begin{lemma}\label{lem:checkingHallCondition_graphversion_1}
 	Let $G=(A \cup B,E)$ be a bipartite graph with $\lvert A \rvert = \lvert B \rvert = 2m+1$ and minimum degree $m\ge 1$.
 	Then for every $A_1 \subseteq A$, we have that $\lvert N(A_1) \rvert \ge \lvert A_1 \rvert$ except possibly if $G$ has two disjoint induced subgraphs $K_{m,m+1}$ as subgraphs in the following way.
 	There are sets $A_1,A_2,B_1, B_2$ such that $A=A_1 \cup A_2, B=B_1 \cup B_2$ and $\lvert A_1 \rvert = \lvert B_2 \rvert = m+1$ and $\lvert A_2 \rvert = \lvert B_1 \rvert = m$ such that $G[A_1,B_1]$ and $G[A_2,B_2]$ are both isomorphic to $K_{m,m+1}$ and $G[A_1,B_2]$ is an empty graph.
\end{lemma}

\begin{proof}
 Since the minimum degree is $m$, every $A_1 \subseteq A$ with $\abs{A_1} \le m$ satisfies $\lvert N(A_1) \rvert \ge \lvert A_1 \rvert$. 
 On the other hand, since all vertices in $B$ also have minimum degree $m$, whenever $\lvert A_1 \rvert \ge m+2$ and thus $\lvert A \setminus A_1 \rvert < m$, every vertex in $B$ has a neighbour in $A_1$ and thus $N(A_1)=B$ again has size $\geq \lvert A_1 \rvert$.
 As such, the only exception is when $\lvert A_1 \rvert =m+1$ and $\lvert N(A_1)\rvert =m$.
 In that case, denote $A_2=A \setminus A_1$, $B_1=N(A_1)$, and $B_2=B\setminus B_1$.
 By the minimum degree condition and the definitions, we conclude that $G[A_1,B_2]$ is the empty graph and $G[A_1,B_1]$ and $G[A_2,B_2]$ are complete bipartite graphs. Examples are presented in Figure~\ref{fig:Hall_extr_2m+1_1}, where blue dashed edges can be present or not.
\end{proof}

\begin{figure}[ht]
 \centering
 \begin{tikzpicture}
 \foreach \y in {3,4} {
 \foreach \x in {2,3,4}{\draw[thick] (\x,2) -- (\y,0);}
 \foreach \x in {0,1}{\draw[dashed, blue, thick] (\x,2) -- (\y,0);}
 }
 \foreach \y in {0,1,2} {
 \foreach \x in {0,1}{\draw[thick] (\x,2) -- (\y,0);}
 }
	
 \node at (-0.5,2) {$B$};
 \node at (-0.5,0) {$A$};
	
	\foreach \x in {0,1,2,3,4}{\draw[fill] (\x,2) circle (0.1);}
	\foreach \x in {1,2,3,4,5}{\draw[fill] (\x-1,0) circle (0.1);}

	\end{tikzpicture}
 \caption{A bipartite graph $G \subset K_{5,5}$ with minimum degree two and without a perfect matching. \label{fig:Hall_extr_2m+1_1}}
\end{figure}
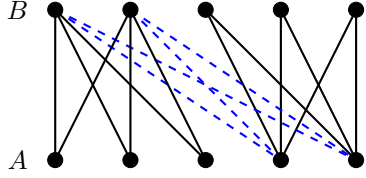

Next, we consider the case where the minimum degree is $m-1$ and the partition classes have size $2m$.

\begin{lemma}\label{lem:checkingHallCondition_graphversion_2}
	Let $G=(A \cup B,E)$ be a bipartite graph with $\lvert A \rvert = \lvert B \rvert = 2m$ and minimum degree $m-1$ for some $m \ge 2$.
	Then for every $A_1 \subseteq A$, we have that $\lvert N(A_1) \rvert \ge \lvert A_1 \rvert$ except possibly for 
	\begin{enumerate}
		\item $\lvert A_1 \rvert =m$ and $\lvert N(A_1) \rvert=m-1$
		\item $\lvert A_1 \rvert =m+1$ and $\lvert N(A_1) \rvert=m-1$
		\item $\lvert A_1 \rvert =m+1$ and $\lvert N(A_1) \rvert=m$
	\end{enumerate}
	Let $A_2=A \setminus A_1$, $B_1=N(A_1)$, and $B_2=B\setminus B_1$. 
	Then we have that $G[A_1,B_2]$ is the empty graph, and respectively the following hold:
	\begin{enumerate}
		\item $G[A_1,B_1]\cong K_{m,m-1}$,
		\item $G[A_1,B_1]\cong K_{m+1,m-1}$ and $G[A_2,B_2]\cong K_{m-1,m+1}$,
		\item $G[A_2,B_2]\cong K_{m-1,m}$
	\end{enumerate}
\end{lemma}

\begin{proof}
 Take an arbitrary subset $A_1 \subseteq A$. 
 If $\abs{A_1} \le m-1$, then due to the minimum degree condition $\delta(G)\ge m-1$, it is immediate that $\lvert N(A_1) \rvert \ge m-1 \ge \lvert A_1 \rvert$.
 In the case $\lvert A_1 \rvert \ge m+2$, and thus $\lvert A \setminus A_1 \rvert < m-1=\delta(G)$, every vertex in $B$ has a neighbour in $A_1$ and thus $N(A_1)=B$, so the condition in Theorem~\ref{thm:Hall_graphversion} again holds.
	 As such, the condition can only not hold when $\abs{A_1} \in \{m,m+1\}$ and $m-1 \le \abs{ N(A_1)}\le \abs{A_1}-1$. 
 The partial characterisation of the extremal graphs is true by the minimum degree condition applied to $A_1,$ $A_1$ and $B_2,$ and $B_2$ respectively.
\end{proof}

\begin{figure}[ht]
 \centering
 \begin{tikzpicture}
 \foreach \y in {3,4,5} {
 \foreach \x in {2,3,4,5}{\draw[thick] (\x,2) -- (\y,0);}
 \foreach \x in {0,1}{\draw[dashed, blue, thick] (\x,2) -- (\y,0);}
 }
 
 \foreach \y in {0,1,2} {
 \foreach \x in {0,1}{\draw[thick] (\x,2) -- (\y,0);}
 }
 
 \draw[red, thick] (4,2) -- (4,0)--(3,2);
 \draw[red, thick] (5,2) -- (5,0);
 
 \node at (-0.5,2) {$B$};
 \node at (-0.5,0) {$A$};
 	
 \foreach \x in {0,1,2,3,4,5} {
 \draw[fill] (\x,2) circle (0.1);
 \draw[fill] (\x,0) circle (0.1);
 }
 \end{tikzpicture}
 \qquad\qquad
 \begin{tikzpicture}
 \foreach \y in {4,5} {
 \foreach \x in {2,3,4,5}{\draw[thick] (\x,2) -- (\y,0);}
 \foreach \x in {0,1}{\draw[dashed, blue, thick] (\x,2) -- (\y,0);}
 }

 \foreach \y in {0,1,2,3} {
 \foreach \x in {0,1}{\draw[thick] (\x,2) -- (\y,0);}
 }

 \node at (-0.5,2) {$B$};
 \node at (-0.5,0) {$A$};
	
	\foreach \x in {0,1,2,3,4,5} {
 \draw[fill] (\x,2) circle (0.1);
 \draw[fill] (\x,0) circle (0.1);
 }
 \end{tikzpicture}
 \caption{Examples of bipartite graphs $G \subset K_{6,6}$ with $\delta(G)=2=m-1$ and without a perfect matching. Red edges present some potential missing edges in a $K_{m,m+1}$. Blue edges present potential edges that can be added.\label{fig:Hall_extr_2m+1}}
\end{figure}
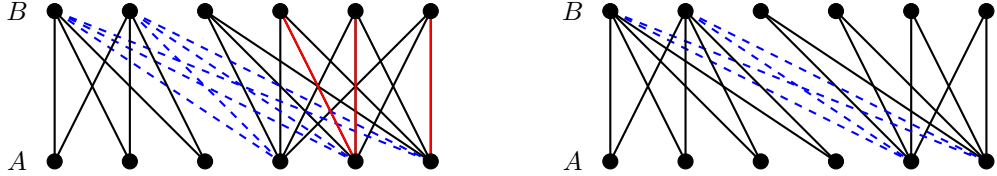

We are now ready to prove Theorem~\ref{thm:maxdegree_improvement}, which we recall states that for $\Delta\ge 4$, if $G$ is a graph of maximum degree $\Delta$ then $\chi^\star_c(G) \leq 2\Delta(G)-2$.

\begin{proof}[Proof of Theorem~\ref{thm:maxdegree_improvement}]
We may assume that $G$ is connected. If $G$ is not $\Delta$-regular, then $\delta^\star(G) \le \Delta-1$, and the theorem is true by~\cite[Thm.~9]{CCDK21}.
So we may assume that $G$ is $\Delta$-regular,
for a fixed $\Delta \ge 4$. 
Let $m=\Delta-1\ge 3$, and $k=2m =2\Delta-2$.
Let $H$ be a $k$-fold correspondence-cover of $G$, via some correspondence-assignment $L:V(G) \to 2^{V(H)}$. 
To be concrete, we label $L(v) = \{1_v , 2_v, \dotsc, k_v\}$ for each $v\in V(G)$ and we drop the subscripts where convenient.
Take an arbitrary edge $uv$ of $G$.
Without loss of generality (one can rename the colours if necessary), we assume that the matching in $H$ between $L(u)$ and $L(v)$ is the ``identity'' connecting $i_u$ to $i_v$ for $1\le i\le k$.

Let $\vec c$ be a correspondence-packing of $G\setminus\{u, v\}$ for the cover graph $H\setminus(L(u)\cup L(v))$, which exists by the non-regular case of the theorem. 
It suffices to extend $\vec c$ to both $u$ and $v$.
We will do so by first choosing $\vec c(u)$ by imposing at most two additional constraints on the choice, and then $\vec c(v).$
For every $1\le i \le k$, let $U_i = L(u)\setminus(\bigcup_{w \in N_G(u)\setminus \{v\}}N_H(c_i(w))$, and define $V_i$ similarly.
Note that the set $U_i$ collects all possible elements of $L(u)$ that can be used for $c_i(u)$ in a proper extension of $\vec{c}$ to $u$ and the same is true for $V_i$. However, some choices of pairs $c_i(u) \in U_i$ and $c_i(v) \in V_i$ may still be in conflict so cannot be used simultaneously for an extension of $\vec{c}$.

We first prove the following claim, that will be useful to show that a proper extension is possible.
\begin{claim}\label{clm:AvoidBadHall}
 Let $G=(A \cup B,E)$ be a bipartite graph with $\lvert A \rvert = \lvert B \rvert = 2m$ and minimum degree $m$ for some $m \ge 3$.
 Let $A=A_1 \cup A_2$ and $B=B_1 \cup B_2$ be partitions such that $\lvert A_1\rvert=m, \lvert B_1\rvert=m-1$,
	$G[A_1,B_1]\cong K_{m,m-1}$, and $G[A_1,B_2]\cong mK_2+K_1$.
	Then for a matching $M$ that contains at most $m-2$ edges of $G([A_1,B_2])$, $G \setminus M$ satisfies the conditions of Theorem~\ref{thm:Hall_graphversion}.
\end{claim}
\begin{claimproof}
In this proof, every neighbourhood will be a neighbourhood in the graph $G \setminus M$, i.e., with $N$ we refer here to $N_{G \setminus M}$.
Let $b$ be the only vertex in $B$ which has no neighbour in $G$ belonging to $A_1$.

Take $B' \subseteq B$. Since $G \setminus M$ has minimum degree $\geq m-1$, we have $\lvert B' \rvert \le \lvert N(B') \rvert$ if $\lvert B' \rvert \not \in \{m,m+1\}$, as explained before in the proof of Lemma~\ref{lem:checkingHallCondition_graphversion_2}.
So now assume that $\lvert B' \rvert \in \{m,m+1\}$.
We consider three cases.

If $\lvert B' \cap B_1 \rvert \ge 2$, then $A_1 \subseteq N(B')$ and every vertex in $B' \cap B_2$ has at least $m-2\ge 1$ neighbours in $A_2$. Thus $\abs{N(B')} \ge \abs{A_1}+1= m+1 \ge \abs{B'}$.

If $b \in B' \subseteq B_2$,
then $\abs{N(b) \cap A_2} \ge m-1$ and $\abs{ N(B')\cap A_1 } \ge \abs{B'} -1-(m-2)$ (by choice of $M$).
So we conclude that $\abs{N(B')} \ge \abs{N(b) \cap A_2} + \abs{ N(B') \cap A_1} \ge \abs{B'}$.

If $b \not\in B' \subseteq B_2$ (so $B'=B_2 \setminus b$), then
$\abs{N(B') \cap A_2} \ge m-2$ and $\abs{ N(B') \cap A_1} \ge m -(m-2)=2$ (at most $m-2$ edges of the matching between $A_1$ and $B'$ are removed) and the conclusion follows again.

In the remaining case, $\lvert B' \cap B_1 \rvert =1$.
The vertex $B' \cap B_1$ has at least $m-1$ neighbours in $A_1$.
The vertices in $B' \cap B_2$ have at least $m-2$ neighbours in $A_2$.
If $\abs{B'}=m$, we conclude since $\abs{N(B')} \ge (m-1)+(m-2) \ge \abs{B'}$.
So we are left with $\abs{B'}=m+1$ and thus either $B'\cap B_2 = B_2 \setminus b$ which implies that $A_1 \subseteq N(B')$,
or $b \in B'$ and $\abs{N(b) \cap A_2} \ge m-1$.
In either case, we have $\abs{N(B')} \ge 2m-2 \ge m+1=\abs{B'}$.
\end{claimproof}

\begin{figure}[ht]
 \centering
 \begin{tikzpicture}
 \foreach \y in {0,1,2} {
 \foreach \x in {0,1}{\draw[thick] (\x,2) -- (\y,0);}
 }

 \foreach \x in {0,1,2,3,4} {
 \foreach \y in {3,4,5} {
 \draw[dashed, blue, thick] (\x,2) -- (\y,0);
 }
 }
 \foreach \x in {0,1,2} {
 \draw[thick, blue] (\x+2,2) -- (\x,0);
 }

 \foreach \x in {0,1,2} {
 \draw[ thick] (\x+3,0) -- (5,2);
 }
	
 \draw [decorate,decoration={brace,amplitude=4pt}] (0.0,2.15)--(1,2.15) node [black,midway,yshift=0.275cm]{$B_1$};
 \draw [decorate,decoration={brace,amplitude=4pt}] (2,-0.15)--(0,-0.15) node [black,midway,yshift=-0.35cm]{$A_1$};
 \draw [decorate,decoration={brace,amplitude=4pt}] (5,-0.15)--(3,-0.15) node [black,midway,yshift=-0.35cm]{$A_2$};

	\node at (5,2.3) {$b$};
 \node at (-0.5,2) {$B$};
 \node at (-0.5,0) {$A$};
	
	\foreach \x in {0,1,2,3,4,5} {
 \draw[fill] (\x,2) circle (0.1);
 \draw[fill] (\x,0) circle (0.1);
 }
 \end{tikzpicture}
 \caption{Sketch of Claim~\ref{clm:AvoidBadHall}, possible edges of $G$.\label{fig:claimrelatedhall}}
\end{figure}
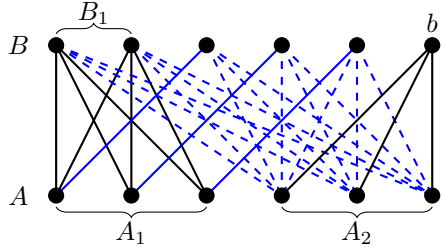

Construct the bipartite graph $G_v$ whose bipartition is $A=(V_i)_{ 1\le i \le 2m}$ and $B=[k]$
and an edge between $V_i$ and $j \in [k]$ if and only $j_v \in V_i$.
Define the bipartite graph $G_u$ analogously.
Since only $\Delta-1$ neighbours of $u$ are already packed, $G_u$ has minimum degree at least $2m-(\Delta-1) =m$, so $G_u$ satisfies Hall's marriage theorem, meaning that $G_u$ contains a perfect matching. This matching corresponds to a choice of $\vec{c}(u)$ which extends the packing $\vec{c}$ to $u$. However, we want to make this choice in a slightly more intricate way, since afterwards we also need to extend $\vec{c}$ to $v$. That is, we do not merely want to find a perfect matching in $G_v$, but rather a perfect matching in $G_v \setminus M$, for some matching $M$ determined by the choice of $\vec{u}$.

If there does not exist a matching $M$ such that $G_v \setminus M$ does not satisfy Hall's marriage theorem, then one can choose $\vec{c}(u)$ and then $\vec{c}(v)$ by assumption, since the latter corresponds to finding a matching in $G_v \setminus M$ for some matching $M$ that is determined by $\vec{c}(u)$.

If there exists a matching $M$ such that $G_v \setminus M$ does not satisfy Hall's marriage theorem, then $G_v\setminus M$ is a bipartite graph with minimum degree $m-1$ for which one of the three cases in Lemma~\ref{lem:checkingHallCondition_graphversion_2} is satisfied.
Up to renaming $A$ and $B$, in all three cases there are partitions $A=A_1 \cup A_2$ and $B=B_1 \cup B_2$ satisfying the conditions of Claim~\ref{clm:AvoidBadHall}.
Choose two specific edges of $G_v[A_1,B_2]$ (which is a matching $mK_2+K_1$), which correspond with pairs $(V_i, x_v), (V_j, y_v)$.
We can impose two additional constraints on the choice for $\vec{c}(u)$; $c_i(u) \not= x_u$ and $c_j(u) \not= y_u$ for some $x_u, y_u \in [k]$ and indices $1 \le i<j \le k$. 
These constraints can be implemented by taking $U'_i=U_i \setminus x$ and $U'_j=U_j \setminus y$ and $U'_{\ell}=U_\ell$ for the remaining indices in $[k]$.
Equivalently, we have deleted the edges $e_1=(U_i, x)$ and $e_2=(U_j, y)$ of $G_u$.
This implies that in both partition classes of $G_u$, at most $2$ vertices have degree equal to $m-1$.
Since in each of the three bad cases in Lemma~\ref{lem:checkingHallCondition_graphversion_2} there are at least $m\ge 3$ vertices in one partition class whose degree is $m-1$, we conclude that $G_u$ always contains a perfect matching.
That is, one can choose $\vec{c}(u)$ with $c_i(u) \in U'_i$ being distinct for every $1 \le i \le k$.
By definition of the $U_i$, this is an extension of the partial packing.
Once $\vec{c}(u)$ has been chosen like this, one can apply Hall's theorem again on $V'_i= V_i \backslash N(c_i(u)), 1 \le i \le k$ to find $\vec{c}(v)$. 
The edges $(V_i, N(c_i(u))\cap L(v))$ for $i \in [k]$ form a matching $M$ for which $G_v\backslash M$ has a perfect matching by Claim~\ref{clm:AvoidBadHall}. The latter perfect matching corresponds with a choice of $\vec{c}(v)$ that extends the partial packing to a
correspondence-packing $\vec{c}$ on $G$.
\end{proof}

Yuster~\cite{Yus21} investigated factors of independent transversals in graphs, and stated a conjecture~\cite[Conj.~1.1]{Yus21} equivalent to the case of complete graphs in our Conjecture~\ref{conj:chic*Delta}\ref{itm:chic*Delta}.
Yuster proved that $\chi_c^\star(K_4)=4$ by computer verification, and stated that the general case of establishing tight upper bounds on $\chi_c^\star(K_n)$ is wide open. 
Theorem~\ref{thm:maxdegree_improvement} immediately gives a tight upper bound for $n=5$. That is, we now know that $\chi_c^\star(K_5)=6$ by a proof that does not involve computer verification.

At this point, we have completed the proofs of all the parts of Theorem~\ref{thm:smalldegrees}. Combined with our previous results we can verify Conjectures~\ref{conj:chil*Delta} for small maximum degree.

\begin{proof}[Proof of Theorem~\ref{thm:smalldegrees}]
	For $\Delta=1$, the bounds follows from Theorem~\ref{thm:degen}.
	For $\Delta \in \{2,3\}$ we need Theorems~\ref{thm:chil*_Delta2}, \ref{thm:chic*_Delta2}, and~\ref{thm:chic*_Delta3}.
	Finally, Theorem~\ref{thm:maxdegree_improvement} gives the upper bound of $\chi_c^\star(G)\le 6$ when $\Delta=4$.
	All of these bounds are sharp due to the complete graph $K_{\Delta+1}$.
\end{proof}

\section{Fractional results}\label{sec:fractional}

We now prove Theorem~\ref{thm:fractionalgreedy}, establishing fractional versions of Conjecture~\ref{conj:chil*Delta}, where we find in part~\ref{itm:chic*Delta} the rounding unnecessary. 
Theorem~\ref{thm:fractionalgreedy} is an immediate corollary of the following generalisation in terms of fractional colouring with local demands, a concept introduced in~\cite{KP18}. 
We do not require a careful discussion of fractional colouring with local demands here, but we point out that a fractional colouring of weight $k$ is equivalent to a probability distribution on independent sets of a graph such that for every vertex, the marginal probability of inclusion in the independent set is at least $1/k$. 
Imposing local demands is a generalisation of this concept equivalent to allowing the required lower bound on the marginal probability to vary according to the vertex, as in the statement below.
We remark that the inductive proof really requires this stronger hypothesis, and note that the argument will not work with degeneracy instead of maximum degree. 
Indeed we observe in Proposition~\ref{prop:chi_ell_bullet_deg+2} that the statement with maximum degree replaced by degeneracy is false.

\begin{lemma} 
Let $G$ be a graph. Consider a correspondence-cover $(L,H)$ of $G$, such that $|L(v)|\ge \deg(v)+1$ for each vertex $v$ of $G$. 
Then there exists a probability distribution on independent sets $I$ of $H$ such that for every vertex $v$ of $G$ and every vertex $x\in L(v)$ of $H$, we have $\Pr(x \in I) \geq 1/ |L(v)|$.
\end{lemma}
\begin{proof}
By adding edges to the cover $H$ if necessary, we may assume that for every edge $uv$ of $G$, the matching between $L(u)$ and $L(v)$ is maximum, i.e.\ of size $\min \{|L(u)|, |L(v)|\}$.
We proceed by induction on the number of vertices of $G$. 
The base case $G = \emptyset$ holds vacuously.

For the induction step, we take a vertex $v$ of $G$ whose list $L(v)$ has maximum size among all vertices of $G$.
Pick a uniformly random vertex $x \in L(v)$. By induction, there exists a random independent set $I_x$ in $H-N[x] $ which satisfies the lemma with respect to the reduced lists $(L(w)-N[x])_{w\in V(G)}$ obtained after removing $N[x]$.
Note that $H-N[x]$ is a valid correspondence-cover for $G-v$ via the map $L$ such that the list of each vertex is large enough, as any list which decreased in size decreased in size by exactly $1$, but the vertices whose lists decreased in size lost the neighbour $v$. 
That is, the conditions are satisfied because $|L(w)-N[x]|\geq |L(w)|-1 \geq \deg(w) = \deg_{G-v}(w)+1$ for every neighbour $w$ of $v$, while still $|L(w)-N[x]|=|L(w)|\geq \deg(w)+1=\deg_{G-v}(w)+1$ for every non-neighbour $w$ of $v$.

The union of $x$ and $I_x$ is the claimed random independent set $I$. Let us confirm this for three types of vertices:

\begin{itemize}
\item For every $y \in L(v)$, we have $\Pr( y \in I) = \Pr( y =x) = 1/ |L(v)|$, by definition.

\item If $u \in V(G)- N[v]$ and $y \in L(u) $, then $\Pr(y \in I) \geq 1/ |L(u)|$ is immediate, as this inequality holds conditioned on any choice of $x$.

\item Finally, let $u\in N(v)$ and let $y\in L(u)$. Note that then $|L(u)| \geq \deg(u)+1 \geq 2$.

We consider three cases for $x$: either it is adjacent to $y$ itself, it is adjacent to a colour in $L(u)\setminus\{y\}$, or it has no neighbours in $L(u)$. 
If $x\sim y$ then $y$ cannot be in $I$. 
If $x$ is adjacent to a colour in $L(u)\setminus\{y\}$ then $y$ is in $I_x$ and hence $I$ with probability at least $1/(|L(u)|-1)$, and in the case than $x$ has no neighbours in $L(u)$, the probability that $y$ is in $I$ is $1/|L(u)|$. 
Since the matching between $L(u)$ and $L(v)$ is maximum by assumption, and $L(v)$ is a list of maximum size, we have $|L(v) - N(L(u))|=|L(v)|-|L(u)|$.
This means that the probability of the third case satisfies 
\[ \Pr(x \notin N(L(u))) = \frac{|L(v)|-|L(u)|}{|L(v)|}, \] which is not too big. 
Putting the conditional probabilities together, we conclude that
\begin{align*}
\Pr(y \in I) &= \frac{1}{|L(u)|-1}\Pr(x \in N(L(u)-y)) + \frac{1}{|L(u)|}\Pr(x \notin N(L(u))) \\
&\ge
\frac{1}{|L(u)|-1} \cdot \frac{|L(u)|-1}{|L(v)|}
+
\frac{1}{|L(u)|} \cdot \frac{|L(v)|-|L(u)|}{|L(v)|}\\
&=
\frac{1}{|L(u)|}.
\end{align*}
\end{itemize}
As such, we have proved the required lower bound on $\Pr(y\in I)$ for every vertex of $H$, as required.
\end{proof}

In~\cite[Prop.~24]{CCDK21}, we gave a construction of a bipartite graph $G$ with degeneracy $d$ but with $\chi_c^\star(G)=2d$. Proposition~\ref{prop:chicbullet_compbip} shows that for this same graph $G$, $\chi_c^\bullet(G) \le d+1$, raising the question of whether the fractional correspondence packing number can exceed $d+1$ in $d$-degenerate graphs. 
We give an example showing even the fractional list packing number can.
The construction is the one we gave in~\cite[Thm.~25]{CCDK21}, but we we strengthen the analysis to show that in fact for the same graph $\chi_\ell^\bullet(G)\ge d+2$. For convenience, we repeat the construction here.

\begin{proposition}\label{prop:chi_ell_bullet_deg+2}
 For every $d \ge 2$,
 there exists a graph $G$ with degeneracy $d$ for which $\chi_\ell^\bullet(G)\ge d+2$.
\end{proposition}
 
\begin{proof}
We will iteratively construct a graph $G$ with $\delta^\star(G)=d$ and a $(d+1)$-list-assignment $L$ such that the covergraph $H$ satisfies $\chi_f(H)>d+1$, implying $\chi_\ell^\bullet(G)\ge d+2$.
We will do so by constructing a sequence of subgraphs $G_1,G_2,\ldots, G$ such that $V(G_1) \subset V(G_2) \subset \ldots \subset V(G)$. 

We start by choosing $G_{1}=(V_1,E_1)=K_{d+1}$ and the associated lists being equal to $[d+1]$ for all vertices.
We now construct $G_2$ by adding a copy $v'$ for every $v \in V_1$ that is connected to all vertices in $V_1 \setminus v$.
Let $V_2=V(G_2) \setminus V_1$ and $L(v')=([d+1]\setminus \{1\}) \cup \{d+2\}$ for every $v'\in V_2$.
Repeating this procedure, in step $m$ we add copies $v'$ for every $v \in V_m$ and connect it to all vertices in $V_m \setminus \{v\}$ and call the set of added vertices $V_{m+1}$.
For $v' \in V_{m+1}$, we let $L(v')=\s_{ij}(L(v))=(L(v) \setminus \{i\} )\cup \{j\} $ for some $i,j \in [d+2]$, i.e.~an $(i,j)$-shift is applied to the lists.
Here we set $V_m=\{v_1^m,\ldots, v_{d+1}^m\}$, where $v_i^{m+1}$ denotes the copy of $v_i^m$.
We choose the shifts to be $\s_{1,d+2},\s_{2,1},\s_{d+2,2}$ in the first three steps.
In general, with a transposition $(i\ j)$
we associate the shifts $\s_{i,d+2},\s_{j,i},\s_{d+2,j}$. 

We repeat the procedure and form the permutation $(d, 1,2,3,\ldots, d-1)$
 
by applying the associated transpositions corresponding to $(1\, 2),(1\, 3),\dotsc,(1\, d)$ in order.
Now continue doing the exact same $3(d-1)$ transpositions another $d-2$ times.
Finally, add a vertex $w$ and connect it to $v_1^{3p(d-1)+1}$ for every $0 \le p \le d-1$, and let $L(w)=[d+1]$ as well.
In all steps, we connected new vertices to exactly $d$ existing vertices and so the degeneracy of the construction satisfies $\delta^{\star}(G)=d$.
Figure~\ref{fig:constr_d2_d+2} gives the construction for $d=2$.

We now analyse the construction, proving the claimed lower bound on the fractional chromatic number.
The plausible $L$-colourings of $G[V_1]$ give a permutation of $[d+1]$.
Fix such a colouring of $V_1$ and consider it a partial $L$-colouring of $G$.
There is one vertex in $V_2$ whose neighbours in $V_1$ are coloured with $[d+1]\setminus \{1\}$ and hence has to be coloured with $d+2$.
The other vertices in $V_2$ have two possible colours, $d+2$ and some $i \in [d+1]\setminus \{1\}$.

A fractional $L$-packing of $G$ is a fractional colouring of weight $d+1$ of the cover graph of $G$ associated with $L$, which corresponds to a random $L$-colouring $c$ of $G$ such that for each vertex $v$ of $G$ and every $x\in L(v)$, $\Pr(c(v)=x)=1/(d+1)$.
Since $L$ gives each vertex in $V_2$ the same list, for each colour $x\ne d+2$, the expected number of vertices in $V_2$ with $c(v)=x$ is at least $1$, and hence the expected number of vertices in $V_2$ which do not get colour $d+2$ is at least $d$. 
This means that the expected number of vertices in $V_2$ which get colour $d+2$ is at most $1$.
Since every $L$-colouring of $G$ has at least one vertex in $V_2$ coloured with $d+2$, we conclude that in fact every $L$-colouring in the fractional packing (i.e.\ which occurs with positive probability) gives exactly one vertex in $V_2$ the colour $d+2$.
This implies that for every colouring in the fractional packing, the colouring restricted to $V_1$ implies the colouring of $V_2$. More specifically, $c(v_i^2)=c(v_i^1)$ if $c(v_i^1)\not=1$ and $c(v_i^2)=d+2$ if $c(v_i^1)=1$,
or equivalently $c(v_i^2)=s_{1,d+2}c(v_i^1)$ for every $i \in [d+1]$.
Furthermore, this observation goes through when comparing partial $L$-colourings of $V_m$ and $V_{m-1}$.
As such, for any colouring $c$ in the fractional packing, $c\left(v_i^{3p(d-1)+1}\right)$ for $0 \le p \le d-1$ either contains all colours in $[d]$, or all of them are equal to $d+1$.
From this, we can conclude that $c(w)=d+1$ or $c(w) \in [d]$ respectively, i.e., $c(v_1^1)=d+1 \Leftrightarrow c(w)\ne d+1$.
The latter implies that the colour $d+1$ appears once on $\{v_1^1,w\}$, while on average it should appear $\frac{2}{d+1}$ times on these two vertices.
Since $d \ge 2$, this is a contradiction.
Hence no fractional $L$-packing exists and thus $\chi_{\ell}^\bullet(G)>d+1$.
\end{proof}

\begin{figure}[ht]
 \centering
 \begin{tikzpicture}
 \definecolor{cv0}{rgb}{0.0,0.0,0.0}
 \definecolor{c}{rgb}{1.0,1.0,1.0}

 \Vertex[L=\hbox{$v_3^1 \colon \{1,2,3\}$},x=3,y=9]{c}
 \Vertex[L=\hbox{$v_2^1 \colon \{1,2,3\}$},x=1.5,y=6]{b}
 \Vertex[L=\hbox{$v_1^1 \colon \{1,2,3\}$},x=3,y=3]{a}

 \foreach \x in {2} {
 \foreach \a in {1,2,3} {
 \Vertex[L=\hbox{$v_\a^\x \colon \{2,3,4\}$},x=3*\x,y=3*\a]{\x_\a};
 };
 }

 \foreach \x in {3} {
 \foreach \a in {1,2,3} {
			\Vertex[L=\hbox{$v_\a^\x \colon \{1,3,4\}$},x=3*\x,y=3*\a]{\x_\a};
 };
 }

 \foreach \x in {4} {
 \foreach \a in {1,2,3} {
 \Vertex[L=\hbox{$v_\a^\x \colon \{1,2,3\}$},x=3*\x,y=3*\a]{\x_\a};
 };
 }
 
 \Vertex[L=\hbox{$w \colon \{1,2,3\}$},x=7.5,y=0]{v};
	\draw[line width=0.65mm, color=blue] (a)--(v)--(4_1);
	\draw[line width=0.65mm, color=blue] (a)--(b)--(c)--(a);
 \draw[line width=0.65mm] (a)--(2_3)--(b)--(2_1)--(c)--(2_2)--(a);

 \draw[line width=0.65mm] (2_3)--(3_1)--(2_2)--(3_3)--(2_1)--(3_2)--(2_3);
 \draw[line width=0.65mm] (3_3)--(4_1)--(3_2)--(4_3)--(3_1)--(4_2)--(3_3);
 \end{tikzpicture}
 \caption{The construction of Proposition~\ref{prop:chi_ell_bullet_deg+2} for $d=2$.\label{fig:constr_d2_d+2}}
\end{figure}
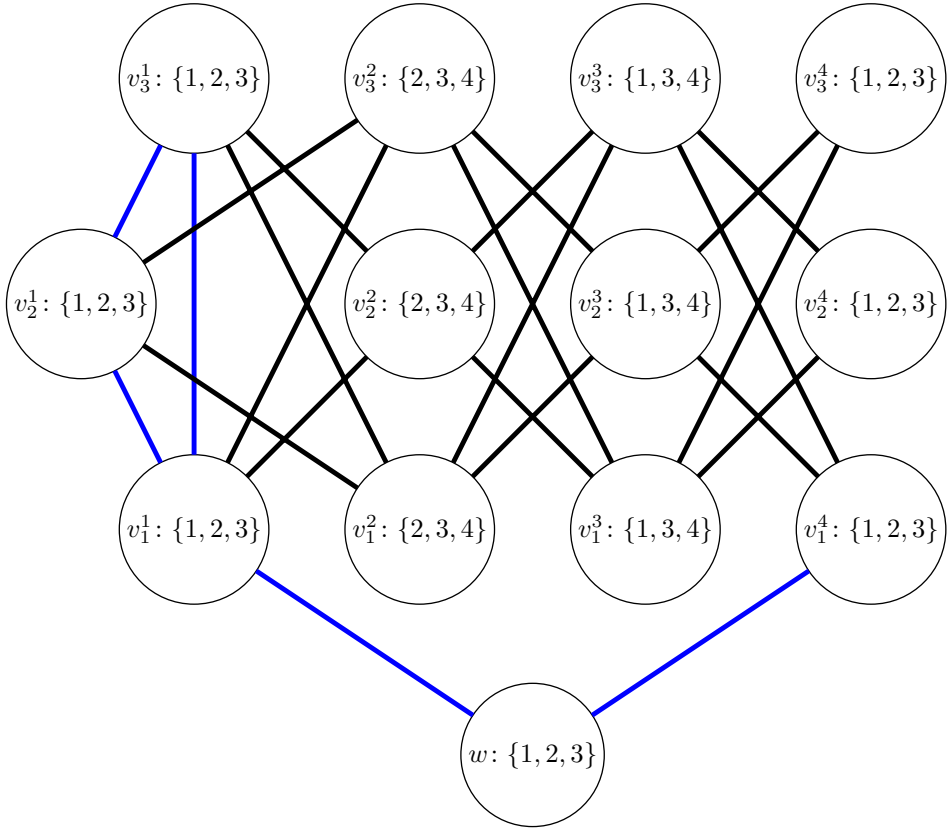

Our next positive result is a version of the greedy bound for bipartite graphs where one is permitted to take the smaller of the maximum degrees over vertices in each part of a bipartition. 
In~\cite[Lem.~33]{CCDK21}, we showed the analogous upper bound for $\chi_\ell^\star$. 
Here, our bound applies to the fractional variant of correspondence packing as well, though the analogue for correspondence packing is false.
In~\cite[Cor.~34]{CCDK21}, we showed that for every complete bipartite graph $K_{a,b}$ with $a>b^b$, we have $\chi_{\ell}^{\star}(K_{a,b})=b+1$ while $\chi_c^{\star}(K_{a,b})=2b$. 
This demonstrates a constant factor gap between list and correspondence packing numbers. 
The proposition below implies that $\chi_{\ell}^\bullet(K_{a,b})=\chi_c^\bullet(K_{a,b})=b+1$ for such $a$ and $b$, demonstrating that the striking factor $2$ difference between list packing and correspondence packing in that construction disappears in the fractional relaxation.

\begin{proposition}\label{prop:chicbullet_compbip}
 Let $G=(A \cup B, E)$ be a bipartite graph with parts $A$ and $B$ having maximum
degrees $\Delta_A$ and $\Delta_B$, respectively, where $\Delta_A \le \Delta_B$. 
Then $\chi_\ell^\bullet(G) \le \chi_c^\bullet(G) \le \Delta_A+1$.
\end{proposition}

\begin{proof}
Consider a correspondence-cover $(L,H)$ of $G$ such that for all vertices $v$ of $G$, $|L(v)|=\Delta_A+1$.
It is sufficient to prove the statement under the assumption that every matching in the correspondence-cover is a perfect matching.
To bound the fractional chromatic number of $H$, we construct a random (maximum) independent set $I=I_B \cup I_A$ of $H$ as follows.
Let $I_B$ contain for each vertex $b\in B$ a uniform random colour $x_b\in L(b)$, chosen independently.
Having fixed a choice of $I_B$, we now choose $I_A$.
Each vertex $a\in A$ has at most $\Delta_A= k-1$ neighbours in $B$, so at least one colour in $L(a)$ is non-adjacent to $I_B$. 
Then we may choose a uniformly random colour $x_a$ from $L(a)\setminus N(I_b)$, and include it in $I_A$. 

Note that for any $a\in A$, the subgraph of $H$ induced by $L(a)$ and $\bigcup_{b \in N(a)}L(b)$ is isomorphic to the cartesian product of a complete graph $K_k$ and a star with $|N(a)|$ leaves. 
By the symmetry of this graph and how $I_b$ is chosen at random, for each $a\in A$ every $x_a\in L(a)$ is in $I$ with the same probability. 
Since the size of the intersection $|I_A\cap L(a)|$ is always exactly $1$, taking expectations we have $\Pr(x_a\in I)=1/k$ for all $a\in A$ and $x_a\in L(a)$. 
We conclude that each vertex of $H$ is in $I$ with probability exactly $\frac 1k$, so $\chi_f(H) \le k = \Delta_A+1$. 
\end{proof}

We conclude this section by noting that the two fractional packing numbers can be different from both the chromatic and integral packing numbers.

\begin{proof}[Proof of Proposition~\ref{prop:ineqs}] 
 We give examples for each case that the quantities can be different.
 \begin{itemize}
 \item Every even cycle $C_{2n}$ satisfies \[ \chi_\ell(C_{2n})=2<3=\chi_\ell^\bullet(C_{2n})=\chi_\ell^\star(C_{2n}). \]
 The list chromatic number of even cycles has been known since the initial study~\cite{ERT80}.
 By Theorem~\ref{thm:chil*_Delta2}, $\chi_\ell^\star(C_{2n})=3$. Moreover, the 2-fold cover of $C_{2n}$ via the list-assignment given in the proof of Theorem~\ref{thm:chil*_Delta2} contains an odd cycle and hence has fractional chromatic number strictly greater than $2$; this proves $3\leq \chi_\ell^\bullet(C_{2n})$. 
 
 \item The fan $F_7$ (formed by adding a universal vertex to a path on 6 vertices) satisfies 
 \[ \chi_\ell(F_7)=\chi_\ell^\bullet(F_7)=3<4=\chi_\ell^\star(F_7). \]
 Note that $K_3$ is a subgraph of $F_7$ to conclude that $3 \le \chi_\ell(F_7)\le \chi_\ell^\bullet(F_7)$.
 A brute-force verification\footnote{\url{https://github.com/StijnCambie/ListPackII}, document F7.py} shows that $\chi_c^\bullet(F_7)=3$, which gives the upper bound $\chi_\ell^\bullet(F_7)\le 3$.
 
 A list-assignment and verification indicating that $\chi_{\ell}^\star(F_7)\ge 4$ is presented in~\cite[Fig.~11.1, Tab.~11.1]{cambie2022thesis}.

 \item The complete bipartite graph $K_{3,3}$ is an example for which \[ \chi_c(K_{3,3})=3<4=\chi_c^\bullet(K_{3,3})=\chi_c^\star(K_{3,3}). \]

 Note that $3=\chi_c(C_4)\le \chi_c(K_{3,3})\le 3$, where the last inequality is true since at most $3\cdot 3!=18$ out of $27$ possible colourings of one partition class cannot be extended to the other partition class (it also follows from Brooks' theorem for correspondence colouring).
 The inequality $3<\chi_c^\bullet(K_{3,3})$ is proved by computer verification\footnote{\url{https://github.com/StijnCambie/ListPackII}, document K3-3.py} and the upper bound $\chi_c^\star(K_{3,3}) \le 4$ is given in Theorem~\ref{thm:chic*_Delta3}.
 \item Any cycle $C_n$ satisfies \[ \chi_c(C_n)=\chi_c^\bullet(C_n)=3<4=\chi_c^\star(C_n). \]
 The equality $\chi_c(C_n)=3$ was observed by Dvo\v{r}\'{a}k and Postle in~\cite{DvPo18}.
 The equality $\chi_c^\bullet(C_n)=3$ follows from Theorem~\ref{thm:fractionalgreedy}.
 The equality $\chi_c^\star(C_n)=4$ is from Theorem~\ref{thm:chic*_Delta2}. \qedhere
 \end{itemize} 
\end{proof}

\section{Concluding remarks}\label{sec:concrem}

A main objective in this paper was to more closely analyse the list packing number in fundamental settings, as a way to gain more intuition into the List Packing Conjecture; this led us naturally to the proposal of Conjecture~\ref{conj:chil*Delta}. We put some evidence towards Conjecture~\ref{conj:chil*Delta} first by confirming it for graphs with small maximum degree. 
Restricted to complete graphs, Conjecture~\ref{conj:chic*Delta}\ref{itm:chic*Delta} coincides with Conjecture 1.1 in~\cite{Yus21}, which remains generally open but was previously verified for up to $4$ vertices. 
Here we confirmed it for the complete graph on $5$ vertices via the more general result that $\chi_c^{\star}(G) \leq 2\Delta-2$ for any graph $G$ with maximum degree $\Delta \geq 4$.

We also proved an approximate version of Conjecture~\ref{conj:chil*Delta} via Theorem~\ref{thm:fractionalgreedy} and the introduction of fractional versions of the list and correspondence packing numbers. 
More generally in combinatorics, fractional packing often serves as a critical component in proving an asymptotically matching bound for the respective integral packing problem. Here though, in the context of correspondence packing, we noticed (see the remarks above and below Proposition~\ref{prop:chicbullet_compbip}) that the fractional and integral value actually can differ by a factor $2$, an intriguing barrier to this approach.

We contend that the determination of $\chi_c^\bullet(G) $ may be an interesting problem in its own right. Just as for the original, integral form of list packing, several problems come to mind, especially the fractional versions of our main conjectures from~\cite{CCDK21}, which we made explicit above in Conjecture~\ref{conj:frac}.
A resolution to such fractional questions could yield interesting insights into the List Packing Conjecture.
An especially appealing problem is to determine an optimal upper bound on the fractional list packing number for planar graphs.

\begin{conjecture}\label{conj:fracplanar}
$\chi_\ell^\bullet(G)\le 5$ for any planar graph $G$.
\end{conjecture}

While we have shown examples of graphs $G$ for which $\chi_c^\star(G)\sim 2\chi_c(G)$~\cite[Prop.~24]{CCDK21}, we actually do not know any graph for which the list packing number is two larger than the list chromatic number.
As such, the following is a natural challenge.

\begin{problem}
Find examples of graphs $G$ for which $\chi_{\ell}^\star(G)>\chi_{\ell}(G)+1$.
\end{problem}

In the following three subsections, we give some remarks related to some interesting further directions one could take to understand the list packing number better.

\subsection{Many list-colourings but no list-packing, even for line graphs}\label{subsec:linegraphs}

By a theorem of Hall~\cite{Hall48}, for a $k$-list-assignment $L$ of $K_k$, there are at least $k!$ proper $L$-colourings of $K_k$.
Because there are so many $L$-colourings of $K_k$, the fact that there exists a packing of $k$ \emph{disjoint} $L$-colourings might not seem especially surprising.
Nevertheless, a packing of colourings does not necessarily follow from a large number of colourings.

Consider the Latin square $K_n \binsq K_n$ as the graph with the cells of a $n \times n$ grid as its vertices, where two vertices are adjacent if they are in the same row or column.
Denote the $n^2$ vertices of $K_n \binsq K_n$ by pairs in $[n]^2$, where $[n]=\{1,2,\ldots, n\}$.
Let the list-assignment $L$ of $K_n \binsq K_n$ be given by $L((1,i))=[n+1]\setminus \{1\}$ for every $i \in [n]$,
$L((j,1))=[n+1]\setminus \{2\}$ for every $2 \le j \le n$
and $L((j,i))=[n]$ for every $2 \le i,j \le n$.
This assignment for $K_4 \binsq K_4$ is presented in Table~\ref{tab:L(K_4,4)}, while an example for $n=3$ and the general case is also presented in~\cite[Fig.~4.9, Fig~5.1]{Levit18}.

\begin{table}[ht]
\centering
\begin{tabular}{|c|c|c|c|}
 \hline
 $[5]\setminus 1$&$[5]\setminus 2$&$[5]\setminus 2$&$[5]\setminus 2$\\
 \hline
 $[5]\setminus 1$&$[4]$&$[4]$&$[4]$\\
 \hline
 $[5]\setminus 1$&$[4]$&$[4]$&$[4]$\\
 \hline
 $[5]\setminus 1$&$[4]$&$[4]$&$[4]$\\
 \hline
\end{tabular}
\caption{Lists on $K_4 \binsq K_4$ without a packing of colourings.\label{tab:L(K_4,4)}}
\end{table}

Extending an observation by Levit~\cite[Lem.~42]{Levit18}, we prove the following.

\begin{proposition}\label{prop:Dinitzpacking}
 There are $n^{n^2(1-o(1))}$ $L$-colourings of $K_n \binsq K_n$ for the $n$-list-assignment $L$ from above.
 Nevertheless, for every $n\ge 2$, $K_n \binsq K_n$ is not fractionally $L$-packable and thus $\chi_\ell^\star(K_n \binsq K_n) \ge \chi_\ell^\bullet(K_n \binsq K_n)>n$.
\end{proposition}
 
\begin{proof}
 We first prove that there are at least $\frac{n-1}{n} N(n)$ many proper $L$-colourings of $K_n \binsq K_n$,
 where $N(n)=n^{n^2(1-o(1))}$ denotes the number of Latin squares of order $n$ (see~\cite[Thm.~17.2]{vLW}).  
 A Latin square corresponds to a proper $n$-colouring of $K_n \binsq K_n$. 
 Take any such proper $n$-colouring for which $(1,1)$ has not been coloured with $1$.
 The vertex coloured by $1$ in the first column and the vertex coloured by $2$ in the first row (if it is not equal to $(1,1)$) are recoloured with $n+1$.
 Then by definition, we have a proper $L$-colouring of $K_n \binsq K_n$ and the lower bound follows because this recolouring gives an injection from the set of Latin squares which do not have a $1$ in the top-left cell to the collection of proper $L$-colourings of $K_n \binsq K_n$.

 Next, we prove that $K_n \binsq K_n$ is not fractionally $L$-packable (i.e.\ the associated cover has fractional chromatic number strictly larger than $n$). 
 It is enough (by Proposition~\ref{prop:fracfacts}) to show that there does not exist a proper $L$-colouring that colours $(1,1)$ with $n+1$. Note that if such a colouring were to exist, the other vertices of the first column would use every colour in $[n]\setminus \{1\}$, and the other vertices of the first row would use every colour in $[n]\setminus \{2\}$.
 As such, every colour in $[n]$ can be used at most $n-2$ times on the vertex subset $([n]\setminus 1)\times ([n]\setminus 1)$.
 Since $(n-1)^2>n(n-2)$, this implies that the colouring cannot be extended.
\end{proof}

So even while there are about as many proper $L$-colourings of the graph as one can hope for, the graph is not fractionally $L$-packable.

Note that $K_n \binsq K_n$ is the line graph of the complete bipartite graph $K_{n,n}$. It is immediate that $K_n \binsq K_n$ has chromatic number $n$. As proposed by Dinitz and proved by Galvin (see~\cite{Gal95}), the list chromatic number of this graph equals $n$ as well. However, Proposition~\ref{prop:Dinitzpacking} indicates that the packing version of Dinitz' problem behaves differently from the colouring version, as the list packing number of $K_n \binsq K_n$ exceeds $n$.
We ask a question that arises from this observation.

\begin{conjecture}[Packing version of Dinitz's problem]
 For $n \ge 3$, $\chi_\ell^\star(K_n \binsq K_n)=n+1$.
\end{conjecture}

\noindent
A possible approach, which is similar to that used in~\cite{Mudrock22}, is to prove that 
\[ \chi_\ell(K_n \binsq K_n \binsq K_{n+1})=n+1.\]
This would imply the result due to the observation that $\chi_\ell(G \binsq K_m)=m$ implies that $\chi_\ell^\star(G) \le m.$ The latter implication is immediate by choosing the same list for the $m$ copies of a particular vertex.

\subsection{Variants of Brooks' theorem}\label{sec:brooks}

A natural Brooks'-type theorem for list packing is false. 
The diamond $K_4^-$ is formed by removing an edge from the complete graph $K_4$.
The $K_4^-$-necklace $G$, consisting of two $K_4^-$s whose degree $2$ vertices are pairwise connected, is a $3$-regular graph (that is not $K_4$) for which $\chi_{\ell}^\bullet(G)=\chi_\ell^\star(G)=4$. 
See Figure~\ref{fig:necklace} for a $3$-list-assignment $L$ that does not admit a fractional $L$-packing. This can easily be verified by a computer\footnote{\url{https://github.com/StijnCambie/ListpackII}, document K4-Necklace.py}, though manual verification is feasible. 
An interesting feature of this example is that an $L$-colouring extending any single mapping $c(v) = x$ for $x\in L(v)$ exists.

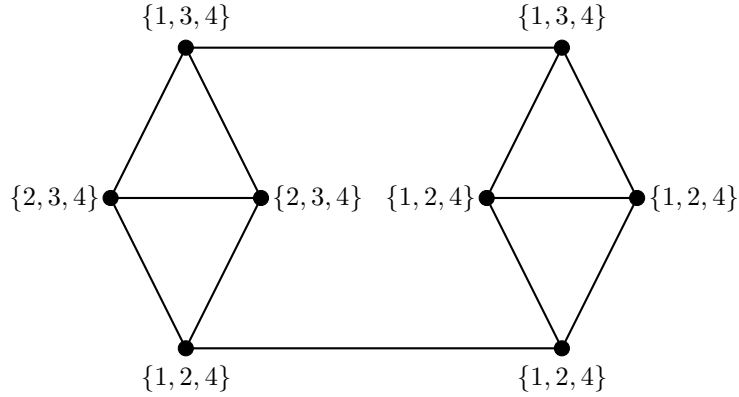
\begin{figure}[ht]
	\centering
	\begin{tikzpicture}

 \foreach \x/\y in {0/0,0/4,-1/2,1/2,5/0,5/4,4/2,6/2} {
 \draw[fill] (\x,\y) circle (0.1);
 }
 
 \draw [thick] (0,0) -- (-1,2)--(0,4)--(1,2)--(0,0)--(5,0)--(4,2)--(5,4)--(6,2)--(5,0);
 \draw [thick] (-1,2)--(1,2);
 \draw [thick] (4,2)--(6,2);
 \draw [thick] (0,4)--(5,4);
 
 \draw (0,4.4) node {$\{1,3,4\}$};
 \draw (0,-0.4) node {$\{1,2,4\}$}; 
 \draw (1.75,2) node {$\{2,3,4\}$}; 
 \draw (-1.75,2) node {$\{2,3,4\}$}; 

 \draw (5,4.4) node {$\{1,3,4\}$};
 \draw (5,-0.4) node {$\{1,2,4\}$}; 
 \draw (3.25,2) node {$\{1,2,4\}$}; 
 \draw (6.75,2) node {$\{1,2,4\}$}; 
 \end{tikzpicture}
 \caption{The $K_4^-$-necklace and a list-assignment for which no (fractional) list-packing is possible.\label{fig:necklace}}
\end{figure}

While we rule out the statement $\chi_\ell^\bullet(G)\le\max\{3,\omega(G),\Delta(G)\}$, which seems an appealing formulation because cycles have list packing number $3$, it is plausible that a Brooks'-type theorem holds with a more esoteric set of exceptional cases. For example, we have not ruled out a bound of the form $\chi_\ell^\bullet(G)\le\max\{4,\omega(G),\Delta(G)\}$, or that there exists an easily-describable set of graphs $\mathcal G$ (including cycles and the $K_4^-$-necklace) such that for connected $G\notin \mathcal G$ we have $\chi_\ell^\star(G)\le\max\{\omega(G),\Delta(G)\}$.
We suggest the following question.

\begin{problem}
 Characterise the graphs of maximum degree $3$ with list packing number $4$.
\end{problem}

For correspondence packing, analogues of Brooks' theorem and Reed's conjecture need to be modified markedly.
We remark\footnote{\url{https://github.com/StijnCambie/ListpackII}, document Petersen.py} that the Petersen graph $P_{5,2}$ satisfies $\chi_c^\star(P_{5,2})\ge 4$, while it is triangle-free and has maximum degree $3$.
The even degree case of Conjecture~\ref{conj:chic*Delta} is the upper bound $\chi_c^\star(G)\le \Delta+2$, and we ask whether $K_5$ the only tight example for $\Delta(G)=4$.

\begin{problem}
 Characterise the graphs of maximum degree $4$ with correspondence packing number $6$.
\end{problem}

Seeking a packing of list-colourings requires a list-assignment with uniform list sizes, but the fractional variant of list packing naturally generalises to list-assignments with arbitrary list sizes. 
For a list-assignment $L$ of $G$, we can ask for a probability distribution on independent sets $I$ of the associated list-cover $(\tilde L,H)$, where $\tilde L(v) = \{i_v : i\in L(v)\}$, such that for each $v\in V(G)$ and $i\in L(v)$, $\Pr(i_v\in I) \ge 1/|L(v)|$ (see~\cite{KP18} for the general theory of fractional colouring with local demands). 
With this, it makes sense to study fractional degree-list-packability as a more structured version of degree-choosability.
A graph $G$ is degree-choosable if, for any list-assignment $L$ such that $|L(v)|=\deg(v)$, $G$ admits an $L$-colouring. 
Erd\H{o}s, Rubin and Taylor~\cite{ERT80}, and independently Borodin~\cite{Bor77} classified the degree-choosable graphs as those which are not \emph{Gallai trees}. 
Here, we note that fractional degree-list-packability can be defined as above for list-assignments with $|L(v)|=\deg(v)$, but point out that the proof for degree-choosability does not extend to this notion because of the $K_4^-$-necklace. 
It is not too hard to come up with irregular examples too, such as the graph $K_4^-$ itself with lists $\{1,2\}$, $\{1,3\}$, $\{1,2,3\}$, $\{1,2,3\}$, and $K_5^-$ with lists $\{1,2,3\}$, $\{1,2,4\}$, and three copies of $\{1,2,3,4\}$. 
We give a computer verification of the former\footnote{\url{https://github.com/StijnCambie/ListpackII}, document K4-Necklace.py} that is easily adapted to give the latter.

\begin{problem}
 Characterise the graphs that are fractionally degree-list-packable.
\end{problem}

\subsection{Algorithms and complexity}

At the heart of many combinatorial problems sits some inherently difficult algorithmic tasks (and {\em vice versa}). The list packing problem is no exception. Furthermore, many intuitively algorithmic tactics we could successfully employ for the corresponding graph colouring problems become blunted in the hunt for list-packings. Most especially, local modifications of the choice applied at one particular vertex or in its neighbourhood become harder to reason about.
We therefore believe that the algorithmic aspects of list packing are a promising research line, and here we make some basic comments based on our results. Further study would be interesting; in particular, the nature of the list packing number makes it natural to explore various classes of graphs, as is common in algorithmic graph theory.

The decision problem associated to list colouring is `graph $k$-list colouring', where an instance is a graph $G$ and we must decide whether $\chi_\ell(G)\le k$. 
We can define an analogous problem `graph $k$-list packing', and relate its complexity to the list colouring problem. 
It is well-known that for $k\ge 3$, `graph $k$-list colouring' is complete for the complexity class $\Pi_2^\mathsf{p}=\mathsf{coNP}^{\mathsf{NP}}$ (in the second level of the polynomial hierarchy) of problems for which a Turing machine with access to an oracle for $\mathsf{NP}$ can verify certificates for `no' instances in polynomial time~\cite{ERT80,GT09}. 
By the classification of languages in $\Pi_2^\mathsf{p}$ according to a description in terms of quantified Boolean formulae, i.e.\ $L\in\Pi_2^\mathsf{p}$ if and only if there is a constant $c$ and language $R\in \mathsf{P}$ such that 
\[ L = \{x : \forall y \exists z : |y|,|z| \le |x|^c,\text{ and }(x,y,z) \in R\}, \]
it is easy to see that graph $k$-list packing lies in the same complexity class $\Pi_2^\mathsf{p}$. 
That is, although combinatorially it seems harder to find a packing than a single list-colouring, there is no difference in computational complexity (at this resolution).
We raise the question of completeness, and the fact that $\Pi_2^\mathsf{p}$-completeness for the closely related `strong $k$-colouring' decision problem is open~\cite{SU02}.

\begin{question}
 Is there some $k_0$ such that for all $k\ge k_0$, graph $k$-list packing is complete for the complexity class $\Pi_2^\mathsf{p}$? Is $k_0=3$?
\end{question}

\noindent
The classification of graphs of list chromatic number $2$ in~\cite{ERT80} shows that graph $2$-list colouring is in $\mathsf{P}$. Similarly, Theorem~\ref{thm:2packable} shows that graph $2$-list packing is in $\mathsf{P}$. 
In much the same way, the fact that for all $k$, graph $k$-list packing restricted to instances of maximum degree $2$ is in $\mathsf{P}$ follows from Theorems~\ref{thm:2packable} and~\ref{thm:cycles}.
The correspondence version of the above question is also natural.

Algorithms that construct in polynomial time list-colourings, list-packings, and more general objects such as independent transversals and strong colourings, have been studied for some time (e.g.~\cite{GHH21,GH20a,Har16}). 
We first observe that many of our results give linear-time constructions.

\begin{remark}
 The proofs in Sections~\ref{sec:paths+cycles},~\ref{sec:subcubic} and~\ref{sec:maxdegree} give rise to constructions of the desired packings with algorithms that run in linear time (as a function of the order $n$ of $G$ and supposing that the maximum degree is fixed).
 Note that Claim~\ref{clm:edgenotinK3} actually implies that there is an edge not belonging to a triangle near every vertex of a $3$-regular graph that is not $K_4$.
 Knowing that there do exist linear time algorithms~\cite{MB83} (again, assuming that the maximum degree is fixed) to derive the degeneracy ordering of $G \setminus v$ and also finding the greedy packing happens in linear time, we conclude easily. 
 Completing the packing is something that happens locally.
\end{remark}

One of the most general algorithmic results, due to Graf and Haxell~\cite[Cor.~26]{GH20a}, can be used to construct list- and correspondence-packings of graphs of maximum degree $\Delta$ and when lists are of size at least $3\Delta+1$ (see~\cite{ABZ07} for the key method that gives a non-constructive version of this result). 
Their result actually applies to strong colouring, which is considerably more general than correspondence packing, see e.g.~\cite{CCDK21}. 
Here, it suffices to note that finding a $k$-colouring of a $k$-fold correspondence-cover $(L,H)$ of a graph $G$ of maximum degree $\Delta$ is a special case of finding a strong $k$-colouring of a graph $H^\star$ of maximum degree $\Delta$. This is by removing the cliques on the lists in $H$ and finding a strong colouring with respect to the partition induced by $L$.
We showed before~\cite[Thms.~3 and~10]{CCDK21} that the method of~\cite{ABZ07} applied in the context of list and correspondence packing gives the bounds 
\begin{align*}
\chi_\ell^\star(G) &\le 1+\Delta(G)+\chi_\ell(G),&
\chi_c^\star(G) &\le 1+\Delta(G)+\chi_c(G),&
\end{align*}
though we did not give constructive proofs of these bounds. 
Somewhat interestingly, we note that Theorem~\ref{thm:degen} gives a polynomial-time construction for list- and correspondence-packings with lists of size $2\Delta$.
That is, we significantly reduce the required lower bound on the size of the partition classes (equivalent to list size) in one of the results of~\cite{GH20a}, at the considerable cost of requiring that the graph we colour is a cover of some bounded-degree graph.

 {\small
\paragraph{Open access.} For the purpose of open access, a CC BY public copyright licence is applied to any Author Accepted Manuscript (AAM) arising from this submission.}

\bibliographystyle{habbrv}
\bibliography{listpack}

\end{document}